\documentclass[11pt]{amsart}
\usepackage{geometry}                
\usepackage{amssymb,amsmath,amsthm}

\newmuskip\pFqmuskip
\newcommand*\pFq[6][8]{%
  \begingroup 
  \pFqmuskip=#1mu\relax
  \mathcode`\,=\string"8000
  \begingroup\lccode`\~=`\,
  \lowercase{\endgroup\let~}\pFqcomma
  {}_{#2}F_{#3}{\left[\genfrac..{0pt}{}{#4}{#5};#6\right]}%
  \endgroup
}
\newcommand{\pFqcomma}{\mskip\pFqmuskip}

\newcommand*\pPq[6][8]{%
  \begingroup 
  \pPqmuskip=#1mu\relax
  \mathcode`\,=\string"8000
  \begingroup\lccode`\~=`\,
  \lowercase{\endgroup\let~}\pPqcomma
  {}_{#2}\Phi_{#3}{\left[\genfrac..{0pt}{}{#4}{#5};#6\right]}%
  \endgroup
}
\newcommand{\pPqcomma}{\mskip\pPqmuskip}
\newtheorem{theorem}{Theorem}[section]
\newtheorem{lemma}[theorem]{Lemma}
\newtheorem{prop}[theorem]{Proposition}

\theoremstyle{definition}

\theoremstyle{remark}
\newtheorem{remark}[theorem]{Remark}

\usepackage{mathtools}

\numberwithin{equation}{section}

\newcommand{\be}{\begin{equation}}
\newcommand{\ee}{\end{equation}}
\newcommand{\ba}{\begin{eqnarray}}
\newcommand{\ea}{\end{eqnarray}}
\newcommand{\baa}{\begin{eqnarray*}}
\newcommand{\eaa}{\end{eqnarray*}}
\newcommand{\bea}{\begin{eqnarray*}}
\newcommand{\eea}{\end{eqnarray*}}
\newcommand{\bb}{\begin{thebibliography}}
\newcommand{\eb}{\end{thebibliography}}

      \def\cF{{\mathcal F}}
\def\cG{{\mathcal G}}      
      \def\cL{{\mathcal L}}
\def\cM{{\mathcal M}}

\title[Exceptional Laurent biorthogonal polynomials]
{Exceptional Laurent biorthogonal polynomials through spectral transformations of generalized eigenvalue problems}

\author{Yu Luo}
\address{Department of Mathematics\\
Zhejiang University of Technology\\
Hangzhou 310014, China}
\email{luoyu4304@outlook.com}

\author{Satoshi Tsujimoto}
\address{Department of Applied Mathematics and Physics\\
Graduate School of Informatics, Kyoto University\\
Sakyo-Ku, Kyoto, 606 8501, Japan}
\email{tsujimoto.satoshi.5s@kyoto-u.jp}

\thanks{Yu Luo\\
luoyu4304@outlook.com\\
+86-18720265091\\
Department of Mathematics, Zhejiang University of Technology, Hangzhou 310014, China\\
Satoshi Tsujimoto\\
tsujimoto.satoshi.5s@kyoto-u.jp\\
+81- \\
Department of Applied Mathematics and Physics, Graduate School of Informatics, Kyoto University, 
Sakyo-Ku, Kyoto, 606 8501, Japan
}

\begin{document}

\begin{abstract}
A formulation is given for the spectral transformation of the generalized eigenvalue problem through 
the decomposition of the second-order differential operators. 
This allows us to construct some Laurent biorthogonal polynomial systems with gaps in the degree 
of the polynomial sequence. 
These correspond to an exceptional-type extension of the orthogonal polynomials, 
as an extension of the Laurent biorthogonal polynomials. 
Specifically, we construct the exceptional extension of the Hendriksen-van Rossum polynomials, 
which are biorthogonal analogs of the classical orthogonal polynomials. 
Similar to the cases of exceptional extensions of classical orthogonal polynomials, 
both of state-deletion and state-addition occur. 

\smallskip
\noindent \textbf{\keywordsname}
Generalized eigenvalue problem; 
Differential operators; 
Laurent biorthogonal polynomials; Exceptional Laurent biorthogonal polynomial; 
Hendriksen-van Rossum polynomials. 

\end{abstract}

\maketitle

\section{Introduction}
The eigenvalue problem and its isospectral transformations of operators expressed in terms of second-order derivatives or ($q$-)differences have a variety of known applications, and various concrete examples are known in quantum mechanics, stochastic processes, and so on. 
In particular, they are closely related to the theory of classical orthogonal polynomials (COP), 
which are known to have especially good properties among orthogonal polynomials.
It is known that the orthogonal polynomials which are eigenfunctions of second-order differential or ($q$-)difference operators are given by Askey-Wilson polynomials and their degenerations, and various results are known on the problem of eigenvalue-preserving deformations of them. %
One topic that has attracted much attention in recent research on COP is called exceptional-type extensions, 
in which a sequence of polynomials with degree jumps can provide a basis for an appropriate weighted $L^2$-space. 
The exceptional orthogonal polynomials (XOP) generalize COP by loosening restrictions on their degree sequence. 

During the last decade, extensive efforts have been devoted to many aspects of the theory and applications of XOP. 
After the first examples of XOP were introduced by G\'omez, Kamran, and Milson in 
\cite{GKM09, GKM10_1}, 
their application potential was immediately recognized by many physicists and mathematicians 
\cite{Quesne08, Quesne09, MR09, OS09, OS11}. 
This is because these XOP, which satisfy second-order differential equations, 
can be used to derive new exactly solvable potentials. 
The Darboux transformation plays an important role in the construction of XOP, 
many examples of XOP are obtained by using this method
\cite{GKM10_2, X q-Racah, GGM13, STA}. 
It was further clarified that multiple-step or higher-order Darboux transformations lead to XOP labeled by multi-indices 
\cite{GKM12, OS13}. 
Dur\'an developed another systematic way of constructing XOP by using the concept of dual families of polynomials 
\cite{BispOP, Krall_Hahn, XOPviaKrall}, 
which is different from the method of Darboux transformation. 
Some important properties such as recurrence relations \cite{X-RR, GKK16, O16}, 
zeros \cite{HS12, GMM13, KM15}
and spectral analysis \cite{LLM16}
regarding XOP were also discussed. 
In \cite{X-Bochner}, a complete classification of the continuous XOP which generalizes the classical ones 
(Jacobi, Laguerre, and Hermite polynomials) was addressed. 
One can also refer to \cite{X-Kra} for a detailed discussion of the exceptional Krawtchouk polynomials.  

%
The exceptional-type extensions have so far been limited to orthogonal polynomials based on the eigenvalue problems (EVP), but in this paper we show that they can be applied to biorthogonal polynomials by considering spectral transformations of the generalized eigenvalue problems (GEVP).
First, we study spectral transformations of GEVPs involving second-order differential operators.
This kind of GEVP is closely related to the theory of biorthogonal rational functions (BORF). 
A specialization of BORF is the class of Laurent biorthogonal polynomials (LBP). 
The spectral transformations are based on the decomposition of the second-order differential operators 
in the GEVP, and they can be seen as generalizations of the Darboux transformations. 
Different from the case of COP, when considering BORF both the GEVP and the adjoint GEVP should be discussed. 
The reason is that the biorthogonality relations hold for the eigenfunctions of the GEVP and their biorthogonal partners, 
where the latter can be obtained by applying an operator of the adjoint GEVP to their eigenfunctions 
\cite{BORF_GEVP}. 
Using the results of these spectral transformations we can construct LBP with jumps 
in their degree sequences, that is, the exceptional extensions of LBP. 
Specifically, we will construct the exceptional extensions of 
the Hendriksen-van Rossum (HR) polynomials \cite{HR}, 
which are believed to be the simplest biorthogonal analogs of the COP. 
We will provide the exceptional HR polynomials and their biorthogonal partners obtained from 
single-step Darboux transformation explicitly. 
Their biorthogonality relations and some related issues will also be discussed.

This paper unfolds as follows. 
In section 2, we first give a brief review of the theory of Darboux transformations for standard eigenvalue problems, 
then we extend the scope to the GEVP and introduce the generalized Darboux transformations. 
In particular, results on the generalized Darboux transformations for the adjoint GEVP are also provided here, 
from which the biorthogonal partners and weight functions can be derived. 
The biorthogonality relations between the eigenfunctions of transformed GEVP and the biorthogonal partners are 
obtained by using the results of \cite{BORF_GEVP}.

Section 3 introduces the definition and some important properties of HR polynomials. 
The quasi-polynomial eigenfunctions (which are defined as products of a gauge factor and a polynomial part) 
and eigenvalues associated with the GEVP of HR polynomials are derived here. 
Interestingly enough, we also find the corresponding quasi-Laurent-polynomial eigenfunctions, 
which are products of a gauge factor and a Laurent-polynomial part. 
These eigenfunctions play an important role in the Darboux transformations, 
as we will choose the seed functions for the construction of the exceptional extensions. 
We also provide the eigenfunctions and eigenvalues of the adjoint GEVP explicitly. 

In section 4, we derive the single-step Darboux transformed HR polynomials 
and their biorthogonal partners, as well as the corresponding weight functions by using the results of section 2 and section 3. 
Here we write the polynomials obtained from the transformed eigenfunctions explicitly 
and show the existence of gaps in their degree sequence, which characterizes the exceptional extension. 
We show there are four types of exceptional HR polynomials, 
which are classified concerning the related seed functions. 
The biorthogonality relations of these exceptional HR polynomials are also presented explicitly. 

In section 5, we give some careful observations for the case of multiple-step Darboux transformations, 
and we provide the corresponding eigenfunctions in terms of Wronski determinants 
whose elements are the seed functions and their derivatives.  

Finally, we end this paper with concluding remarks in section 6.

\section{Darboux transformations}
\subsection{Standard eigenvalue problem}
First, we quickly review the Darboux transformations which act as isospectral transformations for 
standard eigenvalue problems, before introducing the ones for generalized eigenvalue problems required for exceptional extensions of biorthogonal rational functions. 

Let us consider the (standard) eigenvalue problem:
\begin{align}
 \cL \psi(z) = \lambda \psi(z),
\label{EVP_0}
\end{align}
where 
$\cL$ is a second-order differential or difference operator. 
Before applying Darboux transformation to (\ref{EVP_0}), 
we need to decompose the operator $\mathcal{L}$ into two first-order operators. 
Below, we will do this for the case of a second-order differential operator.

If $\cL$ in (\ref{EVP_0}) is given by a second-order differential operator 
\begin{align}
 \cL=A(z)\partial^2_z +B(z)\partial_z +C(z) I,
\end{align}
with the identity operator $I$ and the differential operator $\partial_z=\dfrac{d}{dz}$, 
then for any function $\phi(z)$ which is not identically zero, the operator $\cL$ can be decomposed into
\begin{align}
 \cL = \cG^{(0)} \cF^{(0)} + \dfrac{A(z)\phi''(z)+B(z)\phi'(z)+C(z)\phi(z)}{\phi(z)}I, 
\label{standardEVP_00}
\end{align}
where
\begin{align}
 \cG^{(0)} &
= \dfrac{1}{\phi(z)}\left(A(z)\partial_z+B(z)\right)\phi(z)\epsilon(z) I,
\nonumber\\ 
 \cF^{(0)} &= \dfrac{\phi(z)}{\varepsilon(z)}\partial_z\dfrac{1}{\phi(z)} I,
\nonumber
\end{align}
with any arbitrary decoupling factor $\epsilon(z)$.

Furthermore, if $\phi(z)$ is an eigenfunction of the eigenvalue problem \eqref{EVP_0} with the eigenvalue $\kappa$:
\begin{align}
 \cL \phi(z) = \kappa \phi(z).
\end{align}
Then the last term in \eqref{standardEVP_00} can be replaced by a constant multiplication operator that multiplies 
by $\kappa$ as 
\begin{align}
 \cL = \cG^{(0)} \cF^{(0)} + \kappa I.
\label{L:decompose00}
\end{align}

By applying the operator $\cF^{(0)}$ to \eqref{EVP_0} and using \eqref{L:decompose00}, it can be shown that  
\begin{align*}
& \left(\cF^{(0)}\cG^{(0)}  + \kappa I\right) \cF^{(0)}\psi = \lambda \cF^{(0)}\psi.
\end{align*}
As a consequence, we obtain the modified eigenvalue problem which has the same eigenvalue $\lambda$ with \eqref{EVP_0},
\begin{align}
 \widehat{\cL} \widehat \psi=  \lambda \widehat \psi
\label{ModifiedEVP}
\end{align}
where
\begin{align}
 \widehat{\cL}=\cF^{(0)} \cG^{(0)}+\kappa I, \quad
 \widehat{\psi}=\cF^{(0)} \psi.
\end{align}
Here the operator $\cL$ can be expressed using the product of the first-order operators $\cG^{(0)}$ and $\cF^{(0)}$, 
and the transformed operator $\widehat{\cL}$ obtained by interchanging the order of $\cG^{(0)}$ and $\cF^{(0)}$ is 
the eigenvalue-preserving operator. 
It will be shown in the next subsection that the procedure here can be lifted to the case of generalized eigenvalue problems.

\subsection{Generalized eigenvalue problem}

In what follows, we discuss spectral transformations for 
the GEVP, 
which are used to derive an exceptional analog of BORF. 
The basic theory of spectral transformations for the 
GEVP has been given by Zhedanov \cite{BORF_GEVP}, 
where both the GEVP and the adjoint GEVP were considered.
Here we show that a reformulation through a decomposition of operator pencil in a similar manner to 
the procedure in the previous subsection allows for a more convenient discussion 
of the GEVP and the adjoint GEVP.

We start from the GEVP in the form
\begin{align}
\label{GEVP0}
 \cL_1 \psi = \lambda \cL_2 \psi
\end{align}
where $\cL_1, \cL_2$ are both second-order differential operators.
Here $(\lambda, \psi )$ is an eigen-pair of the eigenvalue $\lambda$ and the eigenfunction $\psi$. 
Spectral transformations of (\ref{GEVP0}) can be organized by using the decompositions of the operators $\cL_1$ and 
$\cL_2$, which is analogous to the Darboux-Crum transformation of the (standard) eigenvalue problem in 
the Sturm-Liouville form. 
Below we present the spectral transformations based on the decomposition of second-order differential operators only (since the spectral transformations concerning $q$-differential operators are essentially the same).

For bounded second-order differential operators $\cL_1$ and $\cL_2$ defined by 
\begin{align}
\label{opL1}
& \cL_1 = A_1(z) \partial_z^2 + B_1(z) \partial_z +C_1(z) I,\\
\label{opL2}
& \cL_2 = A_2(z) \partial_z^2 + B_2(z) \partial_z +C_2(z) I,
\end{align}
with functions $A_i(z), B_i(z), C_i(z)$, $i=1,2$, of $z$, 
we can decompose $\cL_1$ and $\cL_2$ into 
\begin{align}
 & \cL_1 = \widetilde{\phi}_1(z)\left(\cG_1 \cF + I \right) ,
\qquad
 \cL_2 = \widetilde{\phi}_2(z)\left(\cG_2 \cF + I\right),
\end{align}
where
\begin{align}
\label{opG1}
 \cG_1 &
= \dfrac{1}{\widetilde{\phi}_1(z)\phi(z)}
\left(A_1(z)\partial_z+B_1(z)\right)\phi(z)\epsilon(z) I\\
\label{opG2}
 \cG_2  &
= \dfrac{1}{\widetilde{\phi}_2(z)\phi(z)}
\left(A_2(z)\partial_z+B_2(z)\right)\phi(z)\epsilon(z) I\\
\label{opF}
 \cF &
=\dfrac{\phi(z)}{\varepsilon(z)}\partial_z\dfrac{1}{\phi(z)} I,
\end{align}
and 
\begin{align}
\label{tphi1}
\widetilde{\phi}_1(z)&=\dfrac{\cL_1\phi(z)}{\phi(z)}=\dfrac{A_1(z)\phi''(z)+B_1(z)\phi'(z)+C_1(z)\phi(z)}{\phi(z)}, \\
\label{tphi2}
\widetilde{\phi}_2(z)&=\dfrac{\cL_2\phi(z)}{\phi(z)}=\dfrac{A_2(z)\phi''(z)+B_2(z)\phi'(z)+C_2(z)\phi(z)}{\phi(z)}.
\end{align}
The GEVP (\ref{GEVP0}) can then be presented as
\begin{align}
\label{GEVP_0_1}
\widetilde{\phi}_1(z)\left(\cG_1\cF +I\right)\psi(z)
= \lambda \widetilde{\phi}_2(z) \left(\cG_2\cF + I\right) \psi(z).
\end{align}

By taking $\phi(z)$ as an eigenfunction with an eigenvalue $\kappa$ such that
\begin{align*}
 \left(\cL_1 -  \kappa \cL_2\right)\phi(z) =0,
\end{align*}
then 
\begin{align}
 \widetilde{\phi}_1(z) = \kappa \widetilde{\phi}_2(z).
\end{align}
Thus the GEVP (\ref{GEVP_0_1}) can be rewritten into 
\begin{align}
 & \kappa \left(\cG_1 \cF + I\right) \psi(z) = 
\lambda  \left(\cG_2 \cF + I\right)\psi(z).
\end{align}
Hereafter, we denote the eigenpairs of the GEVP \eqref{GEVP0} by $(\kappa, \phi(z))$.

Next, we consider the adjoint problem. Let $\cL_1^*$  and $\cL_2^*$ denote 
the adjoint operator of $\cL_1$  and $\cL_2$, respectively, 
then the adjoint problem of (\ref{GEVP0}) is given by
\begin{align}
(\cL_1- \tau \cL_2 )^*\psi^{\ast}(z)=
 (\cL_1^*- \tau \cL_2^* ) \psi^{\ast}(z) = 0,
\label{GEVPad}
\end{align}
where $(\tau, \psi^{\ast}(z) )$ is an eigenpair of the adjoint GEVP. 
Hereafter we denote the adjoint operators of $\cG_j, \cF_j$ by $\cG_j^*, \cF_j^*$, $j=1,2$, respectively.
Again, the adjoint operators $\cL_1^*$ and $\cL_2^*$ can be decomposed into
\begin{align}
\label{L1_ast}
\cL_1^* &= \partial_z^2A_1(z)  - \partial_zB_1(z) + C_1(z)I  
= \left(\cF^* \cG_1^* + I \right) \widetilde{\phi}_1(z) I,\\
\label{L2_ast}
\cL_2^* &=  \partial_z^2A_2(z)  - \partial_zB_2(z) + C_2(z)I  
= \left(\cF^* \cG_2^* + I \right) \widetilde{\phi}_2(z) I,
\end{align}
where
\begin{align}
\label{opaG1}
 \cG_1^* &= \phi(z)\epsilon(z)\left(-\partial_zA_1(z)+B_1(z)\right)\dfrac{1}{\widetilde{\phi}_1(z)\phi(z)}I,\\
\label{opaG2}
 \cG_2^* &= \phi(z)\epsilon(z)\left(-\partial_zA_2(z)+B_2(z)\right)\dfrac{1}{\widetilde{\phi}_2(z)\phi(z)}I,\\
\label{opaF}
 \cF^* &= -\dfrac{1}{\phi(z)} \partial_z \dfrac{\phi(z)}{\epsilon(z)} I.
\end{align}

As a consequence, the GEVP (\ref{GEVP0}) and the adjoint GEVP (\ref{GEVPad}) can be rewritten as
\begin{align}
\label{GEVP_decomp}
 &(\kappa \cG_1-\lambda \cG_2) {\cF} \psi(z) = (\lambda - \kappa)\psi(z), \\
\label{GEVPad_decomp}
 &{\cF^*} (\kappa \cG_1^*-\tau\, \cG_2^*) \widetilde{\phi}_2(z) \psi^{\ast}(z) = (\tau - \kappa){\widetilde{\phi}_2(z)\psi^{\ast}(z)},
\end{align}
respectively.

By introducing the transformed functions
\begin{align}
\label{DT_f1}
 \widehat{\psi}(z)= {\cF}{\psi}(z), \qquad
 \widehat{\psi^{\ast}}(z)= \left(\kappa \cG_1^* - \tau\cG_2^*\right)\widetilde{\phi}_2(z)\,\psi^{\ast}(z), 
\end{align}
the GEVPs (\ref{GEVP_decomp}) and (\ref{GEVPad_decomp}) can be rewritten into 
\begin{align}
\label{backward_hpsi}
 & (\kappa \cG_1-\lambda \cG_2) \widehat{{\psi}}(z) = (\lambda - \kappa) {\psi(z)}, \\
\label{backward_hchi}
 &{\cF^*} \widehat{\psi^{\ast}}(z) = (\tau - \kappa) \widetilde{\phi}_2(z){\psi^{\ast}(z)},
\end{align}
where the left-hand sides can be seen as backward operators from $\widehat \psi(z)$ to $\psi(z)$ 
and from $\widehat \psi^{\ast}(z)$ to $\widetilde\phi_2(z) \psi^{\ast}(z)$, respectively.
Moreover, if we apply $\cF$ and $\kappa \cG_1^*- \tau \cG_2^*$ to 
(\ref{GEVP_decomp}) and (\ref{GEVPad_decomp}), respectively, 
then we obtain
\begin{align}
\label{mGEVP_00}
 & {\cF} (\kappa \cG_1-\lambda \cG_2) \widehat{{\psi}}(z) = (\lambda - \kappa) \widehat{{\psi}}(z), \\
\label{maGEVP_00}
 &(\kappa \cG_1^*-\tau \cG_2^*) {\cF^*}\, \widehat{\psi^{\ast}}(z) = (\tau - \kappa) \widehat{\psi^{\ast}}(z).
\end{align}

Let us introduce the transformed operators defined by
\begin{align}
\label{factor_L0}
&\widehat{\cL}_1 = \kappa\left({\cF} \cG_1 + I\right),
\hspace{6mm} \widehat{\cL}_2 = {\cF} \cG_2 + I,  \\
\label{factor_L1}
&\widehat{\cL}_1^* = \kappa\left(\cG_1^* \cF^*  + I\right),
\hspace{4mm} \widehat{\cL}_2^* = \cG_2^* \cF^* + I. 
\end{align}
Then from (\ref{mGEVP_00}) and (\ref{maGEVP_00}) we obtain 
the transformed GEVPs:
\begin{align}
\label{XL1_ast}
 \left(\widehat{\cL}_1 - \lambda \widehat{\cL}_2\right) \widehat{\psi} = 0, \qquad
 \left(\widehat{\cL}_1^* - \tau \widehat{\cL}_2^*\right) \widehat{\psi^{\ast}} = 0.
\end{align}
The Darboux transformed operators can be written down explicitly as follow. 

\begin{prop}
\label{prop_1step}
The Darboux transformed operators $\widehat{\cL}_1$ and $\widehat{\cL}_2$ can be expressed as 
\begin{align}
&\widehat{\cL}_1 = \kappa\left( \widehat{A}_1(z) \partial_z^2 + \widehat{B}_1(z) \partial_z + \widehat{C}_1(z) I \right)\\
&\widehat{\cL}_2 = \widehat{A}_2(z) \partial_z^2 + \widehat{B}_2(z) \partial_z + \widehat{C}_2(z) I
\end{align}
where for $j=1,2$, 
\begin{align}
& \widehat{A}_{j}(z)
=\dfrac{A_j(z)}{\widetilde{\phi}_j(z)} , \\
& \widehat{B}_{j}(z)
=\dfrac{B_j(z)}{\widetilde{\phi}_j(z)}+2\dfrac{\epsilon'(z)}{\epsilon(z)}\dfrac{A_j(z)}{\tilde{\phi}_j(z)}
+\left(\dfrac{A_j(z)}{\widetilde{\phi}_j(z)}\right)',  \\
& \widehat{C}_{j}(z) 
=1+\left(\dfrac{B_j(z)}{\widetilde{\phi}_j(z)}\right)'
-2\dfrac{A_j(z)}{\tilde{\phi}_j(z)}\left(\dfrac{\phi'(z)}{\phi(z)}\right)^2
+\left(\dfrac{A_j(z)(\epsilon'(z)+\phi'(z))}{\widetilde{\phi}_j(z)}\right)' \\
& \hspace{16.4mm} 
+\dfrac{B_j(z)}{\widetilde{\phi}_j(z)}\left(\dfrac{\epsilon'(z)}{\epsilon(z)}-\dfrac{\phi'(z)}{\phi(z)}\right), \nonumber
\end{align}
and $\widetilde{\phi}_j(z)$, $j=1,2$, are defined by (\ref{tphi1}) and (\ref{tphi2}). 
\end{prop}

\subsection{Biorthogonality and GEVP}

The relationship between the biorthogonality of BORF and the related GEVP 
as well as the adjoint GEVP
has been discussed in \cite{BORF_GEVP}. 
Let $\langle \cdot, \cdot \rangle$ be an inner product. 
From the GEVPs \eqref{GEVP0} and \eqref{GEVPad}, we formally obtain the biorthogonality relation:
\begin{align}
\label{biorth_0}
 \langle \psi, \cL_2^*\psi^{\ast} \rangle = h \delta_{\lambda,\tau},
\end{align}
with some constant $h$. 
In fact, this follows from
\begin{align}
\label{adness}
(\tau-\lambda)\langle \psi, \cL_2^*\psi^{\ast} \rangle
=\langle \psi, (\cL_1^*-\lambda\cL_2^*)\psi^{\ast} \rangle
=\langle (\cL_1-\lambda\cL_2)\psi, \psi^{\ast} \rangle=0. 
\end{align}

Similarly, we can obtain the biorthogonality relation of $\widehat{\psi}$ and $\widehat{\cL}_2^*\widehat{\psi^{\ast}}$:
\begin{align}
\label{Xbiorth}
\langle \widehat{\psi}, \widehat{\cL}_2^*\widehat{\psi^{\ast}}\rangle 
=(\tau - \kappa) \langle \psi, \cL_2^*\psi^{\ast} \rangle
=(\tau - \kappa)  h \delta_{\lambda,\tau} \quad (\mbox{if }\lambda \ne \kappa),
\end{align}
which is shown by
\begin{align}
\label{adness_X}
 \langle \widehat{\psi}, \widehat{\cL}_2^*\widehat{\psi^{\ast}} \rangle
 = \langle \cF \psi, \kappa(\cG_1^* -\cG_2^*)\widetilde{\phi}_2\psi^{\ast} \rangle
 = \langle \psi, (\cL_1^* -\kappa \cL_2^*)\psi^{\ast} \rangle
 = (\tau - \kappa )\langle \psi, \cL_2^*\psi^{\ast} \rangle.
\end{align}
Note that we have used the following relations 
\begin{align*}
& 
\cL_1^* =  \kappa \left(\cF^* \cG_1^* + I \right) \widetilde{\phi}_2 I,\quad
\cL_2^* = \left(\cF^* \cG_2^* + I \right) \widetilde{\phi}_2 I,
\\
&\widehat{\cL}_1^* = \kappa\left(\cG_1^* \cF^* + I\right),
\quad
\widehat{\cL}_2^* = \cG_2^* \cF^* + I, 
\end{align*}
and
\begin{align}
  \widehat{\cL}_2^* \widehat{\psi^{\ast}}
 &= \left(\cG_2^* \cF^* + I\right)\widehat{\psi^{\ast}}
  = (\tau-\kappa)\cG_2^* \widetilde{\phi}_2\psi^{\ast}
+(\kappa\cG_1^* -\tau\cG_2^*)\widetilde{\phi}_2\psi^{\ast}
= \kappa(\cG_1^* -\cG_2^*)\widetilde{\phi}_2\psi^{\ast} .
\end{align}

\section{Hendriksen-van Rossum polynomials}
In this section, we consider a sequence of LBP 
which was first introduced by Askey \cite{A82,A85} in 1982, 
and then discovered again independently by Hendriksen and van Rossum \cite{HR} in 1986. 
This LBP was also mentioned by Gr\"unbaum, Vinet and Zhedanov \cite{CLBP, GVZ04, ABP}. 
Following the notations of \cite{CLBP} we call them the Hendriksen–van Rossum polynomials, 
and HR polynomials for short. 
HR polynomials are believed to be the simplest LBP and can be expressed in terms of the Gauss hypergeometric function: 
\begin{eqnarray}
\label{HR_P}
&&P_n(z):=P_n(z;\alpha,\beta)=\dfrac{(\beta)_n}{(\alpha+1)_n}
~_2F_1\left(
\begin{gathered}
-n, \alpha+1 \\
1-\beta-n
\end{gathered}
;z
\right),  \\
\label{HR_Q}
&&Q_n(z):=Q_n(z;\alpha,\beta)=P_n(z,\beta,\alpha), 
\end{eqnarray}
where $\alpha, \beta$ are real parameters, 
and $\{Q_n(z)\}$ are the biorthogonal partners of $\{P_n(z)\}$. 
The definition of the Gauss hypergeometric function is 
\[
~_rF_s\left(
\begin{gathered}
a_1, a_2, \ldots a_r \\
b_1, b_2, \ldots, b_s
\end{gathered}
;z
\right)=
\sum^{\infty}_{n=0}
\dfrac{(a_1,a_2,\ldots,a_r)_n}{n!(b_1,b_2,\ldots,b_s)_n}z^n, 
\]
where $(a)_n$ is the shifted factorial (or the standard Pochhammer symbol) defined by
\[
(a)_0=1, \quad
(a)_n:=a(a+1)\cdots(a+n-1), \quad n=1,2,\ldots
\]

It is indicated in \cite{HR_1} that HR polynomials are the only LBP which belong to the Hahn class under 
the additional condition that the derivatives $Y'(z)$ of the ``dual'' polynomials $Y_n(z)=P_{n+1}(z)- zP_n(z)$ also form LBP. 
In this sense, HR polynomials can also be called ``classical''. 
Here we introduce some important classical properties of the HR polynomials \cite{A82, A85, HR, HR_1}. 

HR polynomials satisfy the three-term recurrence relation: 
\begin{equation}
\label{TTRR_HR}
P_{n+1}(z)+d_nP_n(z)=z(P_n(z)+b_nP_{n-1}(z)), 
\end{equation}
which can be seen as a GEVP with eigenvalue $z$, 
and  
\begin{equation}
\label{TTRR_Coe_HR}
d_n=-\dfrac{n+\beta}{n+\alpha+1}, \quad
b_n=-\dfrac{n(n+\alpha+\beta)}{(n+\alpha)(n+\alpha+1)}. 
\end{equation}

HR polynomials are polynomial solutions of the linear second-order differential equation: 
\begin{equation}
\label{Deq_HR}
\left[z(1-z)\partial_z^2+(1-\beta-n-(2+\alpha-n)z)\partial_z\right]P_n(z)=-n(\alpha+1)P_n(z). 
\end{equation}
This equation can also be rewritten into a GEVP: 
\begin{equation}
\label{GEVP_HR}
L_1P_n(z)=\theta_nL_2P_n(z), \quad \theta_n=n, 
\end{equation}
with 
\begin{align}
\label{opL}
L_1=A_1(z)\partial_z^2+B_1(z)\partial_z, \quad L_2=B_2(z)\partial_z+C_2(z)I, 
\end{align}
where 
\begin{equation*}
\label{opL_ABC}
A_1(z)=z(1-z), \quad B_1(z)=(1-\beta-(2+\alpha)z), \quad 
B_2(z)=(1-z), \quad C_2(z)=-(\alpha+1). 
\end{equation*}
And it turns out that 
\begin{align}
\label{opL1_P}
&L_1P_n(z)=-n(n+\alpha+1)P_n(z; \alpha+1, \beta-1), \\
\label{opL2_P}
&L_2P_n(z)=-(n+\alpha+1)P_n(z; \alpha+1, \beta-1). 
\end{align}

HR polynomials satisfy a biorthogonality relation on the unit circle: 
\begin{equation}
\label{Biorth_HR}
\dfrac{1}{2\pi i}\int_{|z|=1}P_n(z)Q_m(1/z)w(z)\dfrac{dz}{z}
=\dfrac{1}{2\pi}\int_{0}^{2\pi}P_n(e^{ix})Q_m(e^{-ix})w(e^{ix})dx
=h_n\delta_{mn}, 
\end{equation}
where  
\begin{equation}
\label{w_HR}
w(z)=(-z)^{-\beta}(1-z)^{\alpha+\beta}, \quad
h_n=\dfrac{(\alpha+n+1)_{\infty}(\beta+n+1)_{\infty}}{(n+1)_{\infty}(\alpha+\beta+n+1)_{\infty}},
\end{equation}
and the branch of $(-z)^{-\beta}$ and of $(1-z)^{\alpha+\beta}$ is chosen as follow: 
\begin{eqnarray*}
&&(-z)^{-\beta}=|z|^{-\beta} \text{ if arg}z=\pi (0<\text{arg}z<2\pi), \\
&&(1-z)^{\alpha+\beta}=|1-z|^{\alpha+\beta}\text{ if arg}(1-z)=0 (-\pi<\text{arg}(1-z)<\pi). 
\end{eqnarray*}
An analog of the Christoffel-Darboux formula for LBP is given (\cite{PASI}, Proposition 1.) by 
\begin{equation}
\label{CD_formula}
\dfrac{P_{n+1}(x)Q_n(1/y)-(x/y)^nP_{n+1}(y)Q_n(1/x)}{\tilde{h}_n}
=(x-y)\sum^{n}_{k=0}\dfrac{P_k(x)Q_k(1/y)}{\tilde{h}_k},
\end{equation}
where 
\begin{equation}
\label{CD_formula_h}
\tilde{h}_0=1, \quad
\tilde{h}_n=\prod^{n-1}_{k=0}(1-P_{k+1}(0)Q_{k+1}(0)),  \quad n=1,2,\ldots
\end{equation}
For the HR polynomials, we have 
\[
\tilde{h}_n=\dfrac{(1)_n(\alpha+\beta+1)_n}{(\alpha+1)_n(\beta+1)_n}
=\dfrac{(1)_{\infty}(\alpha+\beta+1)_{\infty}}{(\alpha+1)_{\infty}(\beta+1)_{\infty}}h_n, 
\]
which corresponds to the weight function $w(z)$ with a constant multiplier. 

The positive-definiteness of the linear functional defined by (\ref{Biorth_HR}) is equivalent to $h_n>0, n=0,1,\ldots$, 
which can be guaranteed by a stronger condition: 
\begin{equation}
\label{cond_conv_w}
\alpha>-1, \quad \beta>-1, \quad \alpha+\beta>-1. 
\end{equation}

The moments related to the weight function $w(z)$ are
\[
c_n=\dfrac{1}{2\pi i}\int_{|z|=1}z^nw(z)\dfrac{dz}{z}
=\dfrac{\sin{(\beta\pi)}\Gamma(n-b)}{\sin{((\alpha+\beta)\pi)}\Gamma(-(\alpha+\beta))\Gamma(1+n+\alpha)}, \quad
n=0,1,2,\ldots
\]
under the condition $\Re{(\alpha+\beta)}>-1$, which is equivalent to $\alpha+\beta>-1$ since $\alpha, \beta$ are real. 

We also find the relationship between the coefficients of $L_1$ and the weight function $w(z)$: 
\begin{equation}
\label{w_AB}
\dfrac{w'(z)}{w(z)}=-\dfrac{\alpha z+\beta}{z(1-z)}
=\dfrac{B_1(z)-A'_1(z)}{A_1(z)},
\end{equation}
which can be rewritten into a Pearson equation
\begin{equation}
\label{w_AB_r}
(A_1(z)w(z))'=B_1(z)w(z). 
\end{equation}
This Pearson equation is indeed the same one as in the case of classical orthogonal polynomials. 

Here we give some useful formulas considering the HR polynomials with transformed parameters or their derivatives. 
\begin{lemma}
\label{formula_HR}
The following relations are satisfied by the HR polynomials: 
\begin{align}
\label{formula_HR_1}
&z^nP_n(z^{-1}; \alpha, \beta)=\dfrac{(\beta)_n}{(\alpha+1)_n}P_n(z; \beta-1,\alpha+1), \\
\label{formula_HR_2}
&P'_n(z)=nP_{n-1}(z;\alpha+1,\beta), \\
\label{formula_HR_3}
&\dfrac{P'_n(z;\beta-1,\alpha+1)}{P_n(z;\beta-1,\alpha+1)}
=\dfrac{n}{z}\left[1-\dfrac{(1+\alpha)}{(n-1+\beta)}\dfrac{P_{n-1}(z;\beta-1,\alpha+2)}{P_n(z;\beta-1,\alpha+1)}\right], \\
\label{formula_HR_4}
&\dfrac{P'_n(z; -\alpha-1,-\beta+1)}{P_n(z; -\alpha-1,-\beta+1)}
=\dfrac{n}{z}\left[1-\dfrac{P_{n-1}(z^{-1}; -\beta+1,-\alpha)}{zP_n(z^{-1}; -\beta,-\alpha)}\right].
\end{align}
\end{lemma}

\subsection{Quasi-Laurent-polynomial eigenfunctions associated with HR polynomials}

In the spectral transformations of GEVP introduced in section 2, the function $\phi(z)$ 
which serves as an eigenfunction of the original GEVP plays an important role. 
In fact, one can choose the function $\phi(z)$ as a quasi-polynomial eigenfunction of the GEVP: 
\[
 \left(\cL_1 -  \kappa \cL_2\right)\phi(z) =0, \quad \phi(z)=\xi(z)p(z), 
\] 
where $\xi(z)$ is a gauge function and $p(z)$ is a polynomial. 

\begin{prop}
\label{quasi_ef}
The GEVP (\ref{GEVP_HR}) has two classes of quasi-polynomial eigenfunctions: 
\begin{align}
\label{quasi}
(L_1-\theta^{(j)}_nL_2)\phi^{(j,n)}(z)=0, \quad
\phi^{(j,n)}(z)=\xi_j(z)p^{(j)}_n(z), \quad j=1,2, 
\end{align}
which can be expressed in terms of
\begin{align}
\label{quasi_1}
&\xi_1(z)=1, \quad p^{(1)}_n(z)=P_n(z), \quad \theta^{(1)}_n=n; \hspace{28.5mm} \\
\label{quasi_2}
&\xi_2(z)=(1-z)^{-\alpha-\beta}, \quad 
p^{(2)}_n(z)=P_n(z;-\beta, -\alpha), \quad 
\theta^{(2)}_n=n-\alpha-\beta. 
\end{align}
\end{prop}

\begin{proof}
Suppose that $(\xi(z)p(z), \kappa)$ is an eigenpair of the GEVP (\ref{GEVP_HR}): 
\[
L_1\xi(z)p(z)=\kappa L_2\xi(z)p(z). 
\]
To guarantee that $p(z)$ is a polynomial, it suffices to require that 
\begin{align}
&\xi^{-1}(z)L_1\xi(z)=\eta(z)\left(z(1-z)\partial_z^2+(d+ez)\partial_z+g\cdot I\right), \\
&\xi^{-1}(z)L_2\xi(z)=\eta(z)\left((1-z)\partial_z+h\cdot I\right),
\end{align}
where $d,e,g,h$ are real constants. 
Then we have $\eta(z)=1$, and 
\[
d=1-\beta, \quad
e=\alpha+2h, \quad
\dfrac{\xi'(z)}{\xi(z)}=\dfrac{1+\alpha+h}{1-z},
\]
\[
z(1-z)\dfrac{\xi''(z)}{\xi(z)}=\dfrac{(1+\alpha+\beta)(1+\alpha+h)}{1-z}+g-(2+\alpha)(1+\alpha+h), 
\]
which leads to
\[
\xi(z)=1, \quad e=-2-\alpha, \quad h=-(1+\alpha), \quad g=0, 
\] 
or 
\[
\xi(z)=\dfrac{1}{(1-z)^{\alpha+\beta}}, \quad
e=\alpha+2\beta-2, \quad
h=\beta-1, \quad g=(\alpha+\beta)(1-\beta). 
\]
In the first case, the polynomial part $p(z)$ is an HR polynomial, i.e., $(p(z), \kappa)=(P_n(z), n)$ for 
some non-negative integer $n$. 
In the second case, the polynomial part $p(z)$ is an HR polynomial with the parameters $\alpha, \beta$ replaced by 
$-\beta, -\alpha$, say $p(z)=P_n(z;-\beta,-\alpha)$ for some non-negative integer $n$, 
and the eigenvalue is $n-(\alpha+\beta)$. 
\end{proof}

However, these are not the only eigenfunctions of the GEVP (\ref{quasi}). 
What surprised us was that there exist eigenfunctions 
which are products of a gauge factor and a Laurent polynomial part. 
This kind of eigenfunction has not appeared in the case of COP. 
We call them quasi-Laurent-polynomial eigenfunctions. 

In fact, if we replace $z$ by $z^{-1}$ in the GEVP (\ref{quasi}), then 
\[
(L_1-\theta^{(j)}_nL_2)\phi^{(j,n)}(z)|_{z\rightarrow z^{-1}}=0, \quad j=1,2, 
\]
where the following transformations appear on the left side
\[
\partial_{z^{-1}}=-z^2\partial_z, \quad
\partial_{z^{-1}}^2=z^4\partial_z^2+2z^3\partial_z. 
\]
It turns out that the operator $(L_1-\theta^{(j)}_nL_2)|_{z\rightarrow z^{-1}}$ can be expressed as 
\begin{align}
\label{rela_L_Linv}
\xi_j^{-1}(z^{-1})\left((L_1-\theta^{(j)}_nL_2)|_{z\rightarrow z^{-1}}\right)\xi_j(z^{-1})
=-z\xi_k^{-1}(z)(L_1-\theta^{(k)}_nL_2)\xi_k(z), 
\end{align}
where $j\in\{1,2\}$, $k\in\{3,4\}$, the new gauge factors and new eigenvalues are 
\[
\xi_3(z)=(-z)^{-1-\alpha}, \quad
\theta^{(3)}_n=-n-1-\alpha-\beta,
\]
\[
\xi_4(z)=(1-z)^{-\alpha-\beta}(-z)^{-1+\beta},  \quad
\theta^{(4)}_n=-n-1. 
\]
Therefore, we have obtained two classes of quasi-Laurent-polynomial eigenfunctions of the GEVP (\ref{GEVP_HR}), 
where the Laurent-polynomial parts are $p^{(j)}_n(z^{-1}), j=1,2$. 

To summarize this part, we write down the four classes of quasi-Laurent-polynomial eigenfunctions 
of the GEVP 
\[
(L_1-\theta^{(j)}_nL_2)\phi^{(j,n)}(z)=0, \quad \phi^{(j,n)}(z)=\xi_j(z)p^{(j)}_n(z), \quad
j=1,2,3,4, 
\]
together as follow: 
\begin{align*}
&p^{(1)}_n(z)=P_n(z),&
&\xi_1(z)=1,&  &\theta^{(1)}_n=n; & \\
&p^{(2)}_n(z)=P_n(z; -\beta,-\alpha),& 
&\xi_2(z)=(1-z)^{-\alpha-\beta},& &\theta^{(2)}_n=n-\alpha-\beta; & \\
&p^{(3)}_n(z)=P_n(z^{-1}),& \quad 
&\xi_3(z)=(-z)^{-1-\alpha},& \quad &\theta^{(3)}_n=-n-1-\alpha-\beta; &\\
&p^{(4)}_n(z)=P_n(z^{-1}; -\beta,-\alpha),& \quad 
&\xi_4(z)=(-z)^{-1+\beta}(1-z)^{-\alpha-\beta},& \quad 
&\theta^{(4)}_n=-n-1. &
\end{align*}

\begin{remark}
Note that by deriving (\ref{rela_L_Linv}) we have the relation between the gauge factors 
\[
|\xi_4(z)\xi_1(z^{-1})|=|\xi_3(z)\xi_2(z^{-1})|, 
\]
and the eigenvalues are
\[
\{\theta^{(1)}_n, n=0,1,2,\ldots\}\cup\{ \theta^{(4)}_n, n=0,1,2,\ldots\}=\mathbb{Z}, 
\]
\[
\{\theta^{(2)}_n+\alpha+\beta, n=0,1,2,\ldots\}\cup\{ \theta^{(3)}_n+\alpha+\beta, n=0,1,2,\ldots\}=\mathbb{Z}. 
\]
In view of these relations, we may call $\phi^{(4,n)}(z)$ the dual of $\phi^{(1,n)}(z)$, 
and call $\phi^{(3,n)}(z)$ the dual of $\phi^{(2,n)}(z)$. 
\end{remark}

\subsection{Quasi-Laurent-polynomial eigenfunctions of the adjoint problem}

In view of the biorthogonality of HR polynomials on the unit circle, we read that our argument is $z=e^{ix}$, 
and its Hermitian conjugate is $\overline{z}=e^{-ix}=z^{-1}$. 
Let us define an inner product on the unit circle as follows, 
\begin{align}
\label{inner_0}
\langle f(z), g(z) \rangle
:=\dfrac{1}{2\pi}\int^{2\pi}_{0}f(e^{ix})\overline{g(e^{ix})}dx
=\dfrac{1}{2\pi}\int^{2\pi}_{0}f(e^{ix})g(e^{-ix})dx.
\end{align}
where $f, g\in C^{2}$ are functions with real parameters and satisfy 
\[
\dfrac{1}{2\pi}\int^{2\pi}_{0}|f(e^{ix})|^2dx<\infty, \quad
\dfrac{1}{2\pi}\int^{2\pi}_{0}|g(e^{ix})|^2dx<\infty.
\]
\begin{lemma}
\label{adjoint_ness}
For arbitrary functions 
$A(z)\in C^{2}$, $B(z)\in C^{1}$ and $C(z)$ on the unit circle with real parameters, 
the adjoint of the operator $C(z)I$ is 
\begin{align}
\label{ha_C0}
&
\overline{(C(z)I)^{\ast}}=C(z^{-1})I, 
\end{align}
the formal adjoints of the operators $B(z)\partial_z$ and $A(z)\partial_z^2$ are  
\begin{align}
\label{ha_C}
&
\overline{(B(z)\partial_z)^{\ast}}
=z\partial_z B(z^{-1})z I, \hspace{2mm}
\overline{(A(z)\partial_z^2)^{\ast}}
=z^2\partial_z^2 A(z^{-1})z^2 I. 
\end{align}
\end{lemma}
\begin{proof}
Let $z=e^{ix}$, then we have the following formulas
\begin{align*}
&\hspace{8mm}\partial_z=-ie^{-ix}\partial_x, \hspace{-12mm}
&&\partial_z^2=e^{-2ix}(-\partial_x^2+i\partial_x),& \\
&\hspace{8mm}\partial_x=iz\partial_z,  \hspace{-12mm}
&&\partial_x^2=-z^2\partial_z^2-z\partial_z.&
\end{align*}
For any $f(z),g(z)\in C^2$, it follows from integration by part that 
\begin{align*}
&
\langle f(z), C(z)g(z) \rangle - \langle C(z^{-1})f(z), g(z) \rangle=0, \\
&
\langle f(z), B(z)\partial_zg(z) \rangle - \langle z\partial_z B(z^{-1})zf(z), g(z) \rangle
=\dfrac{1}{2\pi}\bigg[B(e^{-ix})ie^{ix}f(e^{ix}))\overline{g(e^{ix})}
\bigg]\bigg|_{0}^{2\pi}, \\
&
\langle f(z), A(z)\partial_z^2g(z) \rangle-\langle z^2\partial_z^2(A(z^{-1})z^2f(z)), g(z) \rangle \\
&\quad
=\dfrac{1}{2\pi}\bigg[\partial_x(A(e^{-ix})e^{2ix}f(e^{ix}))\overline{g(e^{ix})}
-A(e^{-ix})e^{2ix}f(e^{ix})(\partial_x\overline{g(e^{ix})}+i\overline{g(e^{ix})})
\bigg]\bigg|_{0}^{2\pi}.   
\end{align*}
\end{proof}

Consider the adjoint problem of (\ref{quasi})
\begin{equation}
\label{haGEVP_HR}
(\overline{L_1^{\ast}}-\mu_n\overline{L^{\ast}_2})\phi^{\ast}(z)=0, 
\end{equation}
where 
\begin{align}
\label{ophaL1}
 &\overline{L^{\ast}_1}
 =z^2(z-1)\partial_z^2-z(2+\alpha+(-3+\beta)z)\partial_z+(1-\beta)z I, \\
\label{ophaL2}
 &\overline{L^{\ast}_2}
 =z(z-1)\partial_z+(-1-\alpha+z) I. 
\end{align}

In what follows, we present the quasi-polynomial eigenfunctions for this adjoint problem. 
\begin{prop}
The following two classes of quasi-polynomial eigenfunctions and eigenvalues are solutions to 
the adjoint GEVP (\ref{haGEVP_HR}):
\begin{align}
\label{P_ast}
&(\overline{L_1^{\ast}}-\mu_n\overline{L^{\ast}_2})\phi^{(j,n)\ast}(z)=0, \quad 
\phi^{(j,n)\ast}(z)
=w_j(z)\tilde{p}^{(j)}_n(z), \quad j=1,2, 
\end{align}
where 
\begin{align}
\label{w1}
&w_1(z)=(1-z)^{\alpha+\beta}(-z)^{-1-\alpha}, \quad 
\tilde{p}^{(1)}_n(z)=P_n(z; \beta-1, \alpha+1), \quad 
\mu^{(1)}_n=\theta_n=n, \\
\label{w2}
&w_2(z)=(-z)^{-1-\alpha}, \quad 
\tilde{p}^{(2)}_n(z)=P_n(z;-\alpha-1,-\beta+1), \quad 
\mu^{(2)}_n=n-(\alpha+\beta). 
\end{align}
\end{prop}
Here we omit the proof since it is identical to that of proposition \ref{quasi_ef}. 

Again, we can obtain the quasi-Laurent-polynomials of the adjoint GEVP by substituting the variable $z^{-1}$ for $z$. 
It turns out that
\begin{align}
\label{rela_hL_hLinv}
-zw_j^{-1}(z^{-1})\left((\overline{L^{\ast}_1}-\mu^{(j)}_n\overline{L^{\ast}_2})|_{z\rightarrow z^{-1}}\right)w_j(z^{-1})
=w_k^{-1}(z)(\overline{L^{\ast}_1}-\mu^{(k)}_n\overline{L^{\ast}_2})w_k(z), 
\end{align}
where $j\in\{1,2\}$, $k\in\{3,4\}$, the new gauge factors and new eigenvalues are 
\[
w_3(z)=(1-z)^{\alpha+\beta}(-z)^{-1-\alpha-\beta}, \quad
\mu^{(3)}_n=-n-1-\alpha-\beta,
\]
\[
w_4(z)=z^{-1}, \quad
\mu^{(4)}_n=-n-1,
\]
and the Laurent-polynomial parts are $\tilde{p}^{(j)}_n(z^{-1}), j=1,2$. 

We write down the four classes of quasi-Laurent-polynomial eigenfunctions of the adjoint GEVP 
for reference 
\[
(\overline{L_1^{\ast}}-\mu_n\overline{L^{\ast}_2})\phi^{(j,n)\ast}(z)=0, \quad 
\phi^{(j,n)\ast}(z)
=w_j(z)\tilde{p}^{(j)}(z), \quad j=1,2, 3,4,
\]
where
\begin{align*}
&\tilde{p}^{(1)}(z)=P_n(z; \beta-1, \alpha+1),& \quad 
&w_1(z)=(1-z)^{\alpha+\beta}(-z)^{-1-\alpha},& \quad
&\mu^{(1)}_n=n; & \\
&\tilde{p}^{(2)}(z)=P_n(z;-\alpha-1,-\beta+1),&  \quad 
&w_2(z)=(-z)^{-1-\alpha},& \quad 
&\mu^{(2)}_n=n-\alpha-\beta; & \\
&\tilde{p}^{(3)}(z)=P_n(z^{-1}; \beta-1, \alpha+1),& \quad 
&w_3(z)=(1-z)^{\alpha+\beta}(-z)^{-1-\alpha-\beta}, &\\
&\hspace{4.7mm}\mu^{(3)}_n=-n-1-\alpha-\beta; \nonumber \\
&\tilde{p}^{(4)}(z)=P_n(z^{-1};-\alpha-1,-\beta+1),& \quad 
&w_4(z)=z^{-1},& \quad 
&\mu^{(4)}_n=-n-1. &
\end{align*}

\begin{remark}
One can verify that 
\[
\dfrac{1}{2\pi}\int^{2\pi}_{0}|\phi^{(j,n)}(e^{ix})|^2dx<\infty, \quad
\dfrac{1}{2\pi}\int^{2\pi}_{0}|\phi^{(j,n)\ast}(e^{ix})|^2dx<\infty, \quad j=1,2,3,4, 
\]
under certain conditions. 
In fact, for $j=1,2,3,4$, it holds that
\[
|w_j(z)w_j(z^{-1})|=
|\xi_{5-j}(z; -\alpha-1,-\beta+1)\xi_{5-j}(z^{-1}; -\alpha-1,-\beta+1)|, 
\]
which leads to 
\[
|\phi^{(j,n)\ast}(z)\phi^{(j,n)\ast}(z^{-1})|
=|\phi^{(5-j,n)}(z; -\alpha-1,-\beta+1)\phi^{(5-j,n)}(z^{-1}; -\alpha-1,-\beta+1)|. 
\]
Thus, we have 
\begin{align}
\label{trans_con}
&
\dfrac{1}{2\pi}\int^{2\pi}_{0}|\phi^{(j,n)\ast}(e^{ix})|^2dx
\propto\dfrac{1}{2\pi}\int^{2\pi}_{0}|\phi^{(5-j,n)}(e^{ix}; -\alpha-1,-\beta+1)|^2dx 
\end{align}
So it suffices to show $(2\pi)^{-1}\int^{2\pi}_{0}|\phi^{(j,n)}(e^{ix})|^2dx<\infty$. 
More precisely, we only need to discuss that for $j=1$ and $j=2$, since 
\[
|\phi^{(j,n)}(e^{ix})\phi^{(j,n)}(e^{-ix})|=|\phi^{(j+2,n)}(e^{ix})\phi^{(j+2,n)}(e^{-ix})|, \quad j=1,2. 
\]
For $j=1$, we have 
\begin{align*}
\dfrac{1}{2\pi}\int^{2\pi}_{0}\phi^{(1,n)}(e^{ix})\phi^{(1,n)}(e^{-ix})dx
&
=\left(\dfrac{(\beta)_n}{(\alpha+1)_n}\right)^2
\sum^{n}_{k=0}\dfrac{(-n,\alpha+1)_k}{k!(1-\beta-n)_k}\sum^{n}_{l=0}\dfrac{(-n,\alpha+1)_l}{l!(1-\beta-n)_l} \\
&\hspace{6mm}
\cdot
\dfrac{1}{2\pi}\int^{2\pi}_{0}e^{i(k-l)x}dx \\
&
=\left(\dfrac{(\beta)_n}{(\alpha+1)_n}\right)^2
\sum^{n}_{k=0}\left(\dfrac{(-n,\alpha+1)_k}{k!(1-\beta-n)_k}\right)^2<\infty. 
\end{align*}
For $j=2$, we have 
\begin{align*}
\dfrac{1}{2\pi}\int^{2\pi}_{0}\phi^{(2,n)}(e^{ix})\phi^{(2,n)}(e^{-ix})dx
&
=\left(\dfrac{(-\alpha)_n}{(1-\beta)_n}\right)^2
\sum^{n}_{k=0}\dfrac{(-n,-\beta+1)_k}{k!(1+\alpha-n)_k}\sum^{n}_{l=0}\dfrac{(-n,-\beta+1)_l}{l!(1+\alpha-n)_l} \\
&\hspace{6mm}
\cdot
\dfrac{1}{2\pi}\int^{2\pi}_{0}e^{i(k-l)x}(1-e^{-ix})^{-\alpha-\beta}(1-e^{ix})^{-\alpha-\beta}dx, 
\end{align*}
and 
\[
\dfrac{1}{2\pi}\int^{2\pi}_{0}e^{i(k-l)x}(1-e^{-ix})^{-\alpha-\beta}(1-e^{ix})^{-\alpha-\beta}dx
=(-1)^{-\alpha-\beta}\pi\dfrac{\Gamma(k-l+\alpha+\beta)(1+i\tan((\alpha+\beta)\pi))}
{\Gamma(k-l+1-\alpha-\beta)\Gamma(2(\alpha+\beta))}
\]
under the condition $\Re(\alpha+\beta)<1/2$. 
At the same time, the relation (\ref{trans_con}) requires that $\Re(-\alpha-\beta)<1/2$. 
These together with 
the condition (\ref{cond_conv_w})
yields the stronger condition 
$\alpha>-1$, $\beta>-1$ and 
$-1/2<\alpha+\beta<1/2$. 

\end{remark}

\begin{remark}
\label{eigen_ast}
Notice that the eigenvalues of $\phi^{(j,n)\ast}(z)$ coincide with those of $\phi^{(j,n)}(z)$, i.e., 
$\mu^{(j)}_n=\theta^{(j)}_n$, $j=1,2,3,4$. 
\end{remark}

\begin{remark}
In view of the relations between the gauge factors
\[
|w_4(z)w_1(z^{-1})|=
|w_3(z)w_2(z^{-1})|
\]
and those between the eigenvalues 
\[
\{\mu^{(1)}_n, n=0,1,2,\ldots\}\cup\{ \mu^{(4)}_n, n=0,1,2,\ldots\}=\mathbb{N}, 
\]
\[
\{\mu^{(2)}_n+\alpha+\beta, n=0,1,2,\ldots\}\cup\{ \mu^{(3)}_n+\alpha+\beta, n=0,1,2,\ldots\}=\mathbb{N}, 
\]
we may call $\phi^{(4,n)\ast}(z)$ the dual of $\phi^{(1,n)\ast}(z)$, 
and call $\phi^{(3,n)\ast}(z)$ the dual of $\phi^{(2,n)\ast}(z)$. 
\end{remark}

\begin{remark}
We also notice the following relationship between $w_1(z)$ and the weight function (\ref{w_HR}) of HR polynomials 
\begin{align}
\label{w_1_w}
\overline{w_1(z)}=w_1(z^{-1})=-zw(z), \text{ or }
-zw_1(z)=w(z^{-1})=\overline{w(z)}. 
\end{align}
Thus, it follows from the relation (\ref{opL2_P}) that 
\[
\overline{L^{\ast}_2}\phi^{(1,n)\ast}(z)
=\overline{w(z)}(L_2P_{n}(z))|_{\alpha\rightarrow\beta-1,\beta\rightarrow\alpha+1}
=-(n+\beta)\overline{w(z)}P_n(z; \beta, \alpha), 
\]
where
\[
\overline{L^{\ast}_1}w_1(z)
=\overline{w(z)}(L_1|_{\alpha\rightarrow\beta-1,\beta\rightarrow\alpha+1}), \quad
\overline{L^{\ast}_2}w_1(z)
=\overline{w(z)}(L_2|_{\alpha\rightarrow\beta-1,\beta\rightarrow\alpha+1}). 
\]
Due to the biorthogonality (\ref{biorth_0}), if $\psi(z)$ coincides with $\phi^{(1,m)}(z)$ (=$P_m(z)$), then
\begin{align}
\label{PLP}
\langle \phi^{(1,m)}(z), \overline{L^{\ast}_2}\phi^{(1,n)\ast}(z) \rangle
=-(n+\beta)\langle P_n(z), \overline{w(z)}P_n(z; \beta, \alpha) \rangle
=-(n+\beta)h_n\delta_{mn}. 
\end{align}
In this way, we have verified the fact that the biorthogonal partner of $P_n(z)$ is $Q_n(z)=P_n(z; \beta, \alpha)$, 
where $h_n$ is defined by (\ref{w_HR}). 
\end{remark}

\begin{remark}
The relationship between $w_1(z)$ and the coefficients $A_j(z)$, $B_j(z)$, $C_j(z)$ 
of the operators $L_1$ and $L_2$ can be expressed by the following formula: 
\begin{align}
\label{w_1_f1}
&\dfrac{w'_1(z)}{w_1(z)}=-\dfrac{1}{z}-\dfrac{B_1(z^{-1})-A'_1(z^{-1})}{z^2A_1(z^{-1})}. 
\end{align}
\end{remark}

\section{Exceptional Hendriksen-van Rossum polynomials}
In this section, we will apply spectral transformations to HR polynomials 
and construct the corresponding exceptional extensions. 
First, we shall give a slightly modified version of the Darboux transformation of GEVP compared with 
what was introduced in section 2. 
Since our variable on the unit circle is $z=e^{ix}$, we will change the operators involved in the adjoint problem into 
Hermitian conjugations with respect to a complex variable like we did in subsection 3.2. 
It is not difficult to verify that the equations (\ref{L1_ast})-(\ref{XL1_ast}) still hold after replacing 
$L^{\ast}_1$, $L^{\ast}_2$, $\mathcal{F}^{\ast}$, $\mathcal{G}^{\ast}_1$, $\mathcal{G}^{\ast}_2$
by $\overline{L^{\ast}_1}$, $\overline{L^{\ast}_2}$, 
$\overline{\mathcal{F}^{\ast}}$, $\overline{\mathcal{G}^{\ast}_1}$, $\overline{\mathcal{G}^{\ast}_1}$, 
where
\begin{align}
\label{XL1}
&\overline{L^{\ast}_1}
=z^2\partial_z^2(A_1(z^{-1})z^2 I)+z\partial_z B_1(z^{-1})z I+C_1(z^{-1})I, \\
\label{XL2}
&\overline{L^{\ast}_2}
=z^2\partial_z^2(A_2(z^{-1})z^2 I)+z\partial_z B_2(z^{-1})z I+C_2(z^{-1})I, \\
\label{XF}
&\overline{\mathcal{F}^{\ast}}
=z\partial_z\dfrac{z}{\epsilon(z^{-1})} I-\dfrac{\phi'(z^{-1})}{\epsilon(z^{-1})\phi(z^{-1})} I, \\
\label{XG1}
&\overline{\mathcal{G}^{\ast}_1}
=z\partial_z\dfrac{zA_1(z^{-1})\epsilon(z^{-1})}{\widetilde{\phi}_1(z^{-1})} I
+\epsilon(z^{-1})\dfrac{B_1(z^{-1})+A_1(z^{-1})\dfrac{\epsilon'(z^{-1})}{\epsilon(z^{-1})}
+A_1(z^{-1})\dfrac{\phi'(z^{-1})}{\phi(z^{-1})}}{\widetilde{\phi}_1(z^{-1})} I, \\
\label{XG2}
&\overline{\mathcal{G}^{\ast}_2}
=z\partial_z\dfrac{zA_2(z^{-1})\epsilon(z^{-1})}{\widetilde{\phi}_2(z^{-1})} I
+\epsilon(z^{-1})\dfrac{B_2(z^{-1})+A_2(z^{-1})\dfrac{\epsilon'(z^{-1})}{\epsilon(z^{-1})}
+A_2(z^{-1})\dfrac{\phi'(z^{-1})}{\phi(z^{-1})}}{\widetilde{\phi}_2(z^{-1})} I. 
\end{align}
Specifically, for HR polynomials, we can write the above operators explicitly with
\[
A_1(z)=z(1-z), \quad
B_1(z)=1-\beta-(2+\alpha)z, \quad
C_1(z)=0, 
\]
\[
A_2(z)=0, \quad
B_2(z)=1-z, \quad
C_2(z)=-(1+\alpha). 
\]

Recall that in section 2 we have already shown that the Darboux transformed functions 
$\widehat{\psi}(z)$ and $\widehat{\cL}_2^*\widehat{\psi^{\ast}}(z)$ satisfy the biorthogonality (\ref{Xbiorth}). 
If we choose the eigenfunction $\psi(z)$ as an HR polynomial $P_n(z)$, 
and choose the seed function $\phi(z)$ as a quasi-Laurent-polynomial eigenfunction, i.e., 
\[
\psi(z)=\phi^{(1,n)}(z)=P_{n}(z), \quad 
\lambda=\theta_n, 
\]
\[
\phi(z)=\phi^{(j_0,l_0)}(z)=\xi_{j_0}(z)p^{(j_0)}_{l_0}(z), \quad 
\kappa=\theta^{(j_0)}_{l_0}, 
\]
where $j_0\in\{1,2,3,4\}$ and $l_0\in\mathbb{N}_{\geq 1}, n\in\mathbb{N}_{\geq 0}$, 
then we can construct the Darboux transformed HR polynomials by using the results in section 2. 

First, let us rewrite the Darboux transformed eigenfunction as follows: 
\begin{align}
&\hat{\psi}(z)=\cF\psi(z)=\cF P_n(z)
=\dfrac{1}{\epsilon(z)\phi^{(j_0,l_0)}(z)}
\begin{vmatrix}
\phi^{(j_0,l_0)}(z) & P_n(z) \\
(\phi^{(j_0,l_0)}(z))' & P'_n(z)
\end{vmatrix} \nonumber
\\ &
=\dfrac{1}{\epsilon(z)Q^{(j_0)}(z)p^{(j_0)}_{l_0}(z)}\bigg[Q^{(j_0)}(z)
\left(p^{(j_0)}_{l_0}(z)P'_{n}(z)-(p^{(j_0)}_{l_0}(z))'P_{n}(z)\right)
-P^{(j_0)}(z)p^{(j_0)}_{l_0}(z)P_{n}(z)
\bigg], \nonumber
\end{align}
where the polynomials $Q^{(j_0)}(z)$ and $P^{(j_0)}(z)$ are determined from the numerator and denominator polynomials of $\xi'_{j_0}(z)/\xi_{j_0}(z)$, respectively, as 
\begin{align}
\label{PQ}
\dfrac{\xi'_{j_0}(z)}{\xi_{j_0}(z)}
=\dfrac{P^{(j_0)}(z)}{Q^{(j_0)}(z)}. 
\end{align}
By choosing the decoupling factor $\epsilon(z)$ as 
\begin{align}
\label{eps}
\epsilon^{(j_0,l_0)}(z)=\dfrac{1}{Q^{(j_0)}(z)p^{(j_0)}_{l_0}(z)}, 
\end{align}
the Darboux transformed eigenfunction $\hat{\psi}(z)$ can be normalized to
\begin{align}
\label{hpsi_0}
\hat{\psi}^{(j_0,l_0,n)}(z)
=Q^{(j_0)}(z)
\left(p^{(j_0)}_{l_0}(z)P'_{n}(z)-(p^{(j_0)}_{l_0}(z))'P_{n}(z)\right)-P^{(j_0)}(z)p^{(j_0)}_{l_0}(z)P_{n}(z). 
\end{align}
The right-hand side is a polynomial if $j_0=1,2$ while it is a Laurent polynomial otherwise. 
Specifically, let us write down these four classes of eigenfunctions: 
\begin{align}
&
\hat{\psi}^{(1,l_0,n)}(z)
=P_{l_0}(z)P'_{n}(z)-P'_{l_0}(z)P_{n}(z), \\
&
\hat{\psi}^{(2,l_0,n)}(z)
=(1-z)\left(P_{l_0}(z; -\beta,-\alpha)P'_{n}(z)-P'_{l_0}(z; -\beta,-\alpha)P_{n}(z)\right) \\
&\hspace{24mm}
-(\alpha+\beta)P_{l_0}(z; -\beta,-\alpha)P_{n}(z), \nonumber \\
&
\hat{\psi}^{(3,l_0,n)}(z)
=z\left(P_{l_0}(z^{-1})P'_{n}(z)-(P_{l_0}(z^{-1}))'P_{n}(z)\right) \\
&\hspace{24mm}
+(1+\alpha)P_{l_0}(z^{-1})P_{n}(z),  \nonumber \\
&
\hat{\psi}^{(4,l_0,n)}(z)
=z(z-1)\left(P_{l_0}(z^{-1}; -\beta,-\alpha)P'_{n}(z)-(P_{l_0}(z^{-1}; -\beta,-\alpha))'P_{n}(z)\right) \\
&\hspace{24mm}
-(1-\beta-(1+\alpha)z)P_{l_0}(z^{-1}; -\beta,-\alpha)P_{n}(z).  \nonumber
\end{align}
Both of $\hat{\psi}^{(1,l_0,n)}(z)$ and $\hat{\psi}^{(2,l_0,n)}(z)$ are polynomials, 
as well as $z^{l_0}\hat{\psi}^{(3,l_0,n)}(z)$ and $z^{l_0}\hat{\psi}^{(4,l_0,n)}(z)$. 
Denote them by $P^{(j_0,l_0,n)}(z)$, $j_0=1,2,3,4$, 
and the degrees of them are $n+l_0-1$, $n+l_0$, $n+l_0$ and $n+l_0+1$, respectively. 

\begin{remark}
Since the above eigenfunctions are obtained from $\cF\phi^{(1,n)}(z)$, 
the other eigenfunctions of the Darboux transformed GEVP are expected to be obtained from 
$\cF\phi^{(j,n)}(z)$ for $j=2,3,4$. 
However, $\cF\phi^{(j,n)}(z)$ ($j=2,3,4$) turns out to be rational with the decoupling factor defined by (\ref{eps}). 
Therefore, in the cases of $\cF\phi^{(j,n)}(z)\neq 0$, 
the only (Laurent-)polynomial eigenfunctions of the Darboux transformed GEVP are given by $\cF\phi^{(1,n)}(z)$. 
The cases where $\cF\phi^{(j,n)}(z)=0$ corresponds to $(j_0,l_0)=(j,n)$, 
where $\hat{\psi}^{(1,l_0,l_0)}(z)=0$, and the so-called state-deletion occurs. 
Here the non-vanishing eigenfunction corresponding to the eigenvalue $\kappa=\theta_{l_0}^{(j_0)}$ can be 
chosen from the kernel of the backward operator $\kappa(\cG_1-\cG_2)$ 
(in view of the factorizations of the operators $\widehat{L}_1$ and $\widehat{L}_2$, see (\ref{factor_L0})). 
Denote them by $\hat{\psi}^{(j_0,l_0,j_0,l_0)}(z)$, and it is not difficult to solve the equation 
$\kappa(\cG_1-\cG_2)\hat{\psi}^{(j_0,l_0,j_0,l_0)}(z)=0$. 
Consequently, one obtains 

\begin{align}
&
\hat{\psi}^{(j_0,l_0,j_0,l_0)}(z)=
\dfrac{1}{\epsilon^{(j_0,l_0)}(z)\phi^{(j_0,l_0)}(z)}\exp{\left\{\int\dfrac{\kappa B_2(z)-B_1(z)}{A_1(z)}dz\right\}} \nonumber \\
&\hspace{19mm}
=
\label{v_state}
\begin{cases}
(1-z)^{-1-\alpha-\beta}z^{-1+\beta+l_0}, & j_0=1, \\
-z^{-1-\alpha+l_0}, & j_0=2, \\
(-1)^{1+\alpha}(1-z)^{-1-\alpha-\beta}z^{-l_0}, & j_0=3, \\
(-1)^{-\beta}z^{-l_0}, & j_0=4. 
\end{cases}
\end{align}
We can determine that $\hat{\psi}^{(4,l_0, 4,l_0)}(z)=z^{-l_0}$ (where the constant coefficient is free) 
belongs to the class of $\{\hat{\psi}^{(4,l_0,n)}(z)\}$, more precisely, it is the starting term of this class. 
In fact, if we consider the sequence $\{P^{(4,l_0,n)}(z)\}$ ($\{z^{l_0}\hat{\psi}^{(4,l_0,n)}(z)\}$), 
it stands for the constant term. 
In this case, the state-addition occurs. 
We denote $P^{(4,l_0,-l_0-1)}(z)=1$ to keep $\deg P^{(4,l_0,n)}(z)=n+l_0+1$. 
\end{remark}

\subsection{Biorthogonal partners}
Let us consider the biorthogonal partners. 
For $\psi^{\ast}(z)=\phi^{(1,n)\ast}(z)$, 
it follows from (\ref{backward_hchi}) and the relation (\ref{w_1_f1}) that 

\begin{align*}
&
\widehat{L}_2^{\ast}\widehat{\psi^{\ast}}(z)
=\widehat{L}_2^{\ast}(\theta^{(j_0)}_{l_0}\overline{\cG^{\ast}_1}-\mu^{(1)}_n\overline{\cG^{\ast}_2})
\tilde{\phi}_2(z^{-1})\phi^{(1,n)\ast}(z)
=\theta^{(j_0)}_{l_0}(\overline{\cG^{\ast}_1}-\overline{\cG^{\ast}_2})\tilde{\phi}_2(z^{-1})\phi^{(1,n)\ast}(z) \\
&
=\epsilon(z^{-1})A_1(z^{-1})z^2\left(
\left(\phi^{(1,n)\ast}(z)\right)'
-\left(\dfrac{w'_1(z)}{w_1(z)}+\dfrac{\theta^{(j_0)}_{l_0} B_2(z^{-1})}{z^2A_1(z^{-1})}
-\dfrac{(\phi^{(j_0,l_0)})'(z^{-1})}{z^2\phi^{(j_0,l_0)}(z^{-1})}\right)
\phi^{(1,n)\ast}(z)
\right) \\
&
=\dfrac{(z-1)w_1(z)}{Q^{(j_0)}(z^{-1})p^{(j_0)}_{l_0}(z^{-1})}
\left(
(\tilde{p}^{(1)}_n(z))'
-\left(\ln{\left(z^{\theta^{(j_0)}_{l_0}}\phi^{(j_0,l_0)}(z^{-1})\right)}\right)'
\tilde{p}^{(1)}_n(z)
\right). 
\end{align*}
The expression in the log symbol can be simplified using the formula (\ref{formula_HR_1}): 
\begin{align*}
&
z^{\theta^{(1)}_{l_0}}\phi^{(1,l_0)}(z^{-1})
=\dfrac{(\beta)_{l_0}}{(\alpha+1)_{l_0}}\tilde{p}^{(1)}_{l_0}(z),  \\
&
z^{\theta^{(2)}_{l_0}}\phi^{(2,l_0)}(z^{-1})
=\dfrac{(-\alpha)_{l_0}}{(-\beta+1)_{l_0}}(z-1)^{-\alpha-\beta}\tilde{p}^{(2)}_{l_0}(z), 
\\
&
z^{\theta^{(3)}_{l_0}}\phi^{(3,l_0)}(z^{-1})
=(-1)^{-1-\alpha}\dfrac{(\beta)_{l_0}}{(\alpha+1)_{l_0}}z^{-\beta}\tilde{p}^{(3)}_{l_0}(z), \\
&
z^{\theta^{(4)}_{l_0}}\phi^{(4,l_0)}(z^{-1})
=(-1)^{-1-\alpha}\dfrac{(-\alpha)_{l_0}}{(-\beta+1)_{l_0}}z^{\alpha}(1-z)^{-\alpha-\beta}\tilde{p}^{(4)}_{l_0}(z). 
\end{align*}
which leads to 
\begin{align*}
\left(\ln{\left(z^{\theta^{(j_0)}_{l_0}}\phi^{(j_0,l_0)}(z^{-1})\right)}\right)'
=
\begin{cases}
\dfrac{(\tilde{p}^{(j_0)}_{l_0}(z))'}{\tilde{p}^{(j_0)}_{l_0}(z)}+\dfrac{P^{(j_0)}(z)}{Q^{(j_0)}(z)}, & j_0=1,2, \\
\dfrac{(\tilde{p}^{(j_0)}_{l_0}(z))'}{\tilde{p}^{(j_0)}_{l_0}(z)}+\dfrac{P^{(j_0)}(z)}{Q^{(j_0)}(z)}
+\dfrac{1+\alpha-\beta}{z}, & j_0=3,4. 
\end{cases}
\end{align*}
Thus, we have 
\begin{equation}
\label{hL2psi}
\widehat{L}_2^{\ast}\widehat{\psi^{\ast}}(z)
=
\begin{cases}
\dfrac{(\beta)_{l_0}}{(1+\alpha)_{l_0}}
\dfrac{z^{-l_0}(z-1)w_1(z)}{Q^{(j_0)}(z^{-1})Q^{(j_0)}(z)(p^{(j_0)}_{l_0}(z^{-1}))^2}Q^{(j_0,l_0,n)}(z), & j_0=1, \\
\dfrac{(-\alpha)_{l_0}}{(-\beta+1)_{l_0}}
\dfrac{z^{-l_0}(z-1)w_1(z)}{Q^{(j_0)}(z^{-1})Q^{(j_0)}(z)(p^{(j_0)}_{l_0}(z^{-1}))^2}Q^{(j_0,l_0,n)}(z), & j_0=2, \\
\dfrac{(1+\alpha)_{l_0}}{(\beta)_{l_0}}
\dfrac{z^{-2l_0}(z-1)w_1(z)}{Q^{(j_0)}(z^{-1})Q^{(j_0)}(z)(\tilde{p}^{(j_0)}_{l_0}(z))^2}Q^{(j_0,l_0,n)}(z), & j_0=3, \\
\dfrac{(-\beta+1)_{l_0}}{(-\alpha)_{l_0}}
\dfrac{z^{-2l_0}(z-1)w_1(z)}{Q^{(j_0)}(z^{-1})Q^{(j_0)}(z)(\tilde{p}^{(j_0)}_{l_0}(z))^2}Q^{(j_0,l_0,n)}(z), & j_0=4. 
\end{cases}
\end{equation}
where the formula (\ref{formula_HR_1}) is used, and 
\begin{align*}
&
Q^{(j_0,l_0,n)}(z)=
Q^{(j_0)}(z)\left(\tilde{p}^{(j_0)}_{l_0}(z)\tilde{p}^{(1)}_n(z))'-(\tilde{p}^{(j_0)}_{l_0}(z))'\tilde{p}^{(1)}_n(z)\right) 
-P^{(j_0)}(z)\tilde{p}^{(j_0)}_{l_0}(z)\tilde{p}^{(1)}_n(z), \hspace{2mm} j_0=1,2, \\
&
Q^{(j_0,l_0,n)}(z)=
z^{l_0}Q^{(j_0)}(z)\left(\tilde{p}^{(j_0)}_{l_0}(z)\tilde{p}^{(1)}_n(z))'-(\tilde{p}^{(j_0)}_{l_0}(z))'\tilde{p}^{(1)}_n(z)\right) \\
&\hspace{22mm}
-z^{l_0}P^{(j_0)}(z)\tilde{p}^{(j_0)}_{l_0}(z)\tilde{p}^{(1)}_n(z)
-(1+\alpha-\beta)z^{l_0-1}Q^{(j_0)}(z)\tilde{p}^{(j_0)}_{l_0}(z)\tilde{p}^{(1)}_n(z), 
 \hspace{2mm} j_0=3,4. 
\end{align*}
Again, one can see that $Q^{(1,l_0,n)}(z)$, $Q^{(2,l_0,n)}(z)$, $Q^{(3,l_0,n)}(z)$ and $Q^{(4,l_0,n)}(z)$ 
are polynomials of degrees
$n+l_0-1$, $n+l_0$, $n+l_0$ and $n+l_0+1$, respectively. 

\begin{remark}
The polynomials $Q^{(j_0,l_0,n)}(z)$ are the biorthogonal partners of $P^{(j_0,l_0,n)}(z)$.
In view of the relation (\ref{rela_X1PQ}), we can choose the biorthogonal partner of $P^{(4,l_0,-l_0-1)}(z)=1$
as $Q^{(4,l_0,-l_0-1)}(z)=1$. 
\end{remark}

We have already shown that the degrees of $P^{(j_0,l_0,n)}(z)$ and $Q^{(j_0,l_0,n)}(z)$ are the same. 
In fact, the biorthogonal partners $Q^{(j_0,l_0,n)}(z)$ can be obtained from $P^{(j_0,l_0,n)}(z)$ 
through some transformations on their parameters. 
\begin{prop}
For $j_0\in\{1,2,3,4\}$, we have 
\begin{align}
\label{rela_X1PQ}
Q^{(j_0,l_0,n)}(z)=P^{(j_0,l_0,n)}(z; \beta-1,\alpha+1),
\end{align}
which means $Q^{(j_0,l_0,n)}(z)$ can be obtained from $P^{(j_0,l_0,n)}(z)$ through 
$\alpha\rightarrow\beta-1, \beta\rightarrow\alpha+1$. 
\end{prop}
These relations can be observed easily by comparing the definitions of 
$p^{(j)}_{n}(z)$ and $\tilde{p}^{(j)}_{n}(z)$. 
So, in what follows we will only give the results considering $P^{(j_0,l_0,n)}(z)$, $j_0=1,2,3,4$, 
since those of $Q^{(j_0,l_0,n)}(z)$ follow immediately from (\ref{rela_X1PQ}). 

Denote the degree sequences of $\{P^{(j_0,l_0,n)}(z)\}$ by $S^{(j_0, l_0)}$, then 
\begin{align}
\label{deg_X1PQ_1}
&
S^{(1, l_0)}=\{l_0-1,\ldots, 2l_0-2, 2l_0, \ldots\}, \\
\label{deg_X1PQ_23}
&
S^{(2, l_0)}= S^{(3, l_0)}=\{l_0,l_0+1,\ldots\}, \\
\label{deg_X1PQ_4}
&
S^{(4, l_0)}=\{0, l_0+1, l_0+2, \ldots \}. 
\end{align}
Different from an ordinary LBP or COP, the degree sequences of $P^{(j_0,l_0, j,n)}(z)$ ($Q^{(j_0,l_0, j,n)}(z)$) 
do not start with $0$ but with a degree $\geq 0$ (since $l_0\in\mathbb{N}_{\geq 1}$). 
This is the key characteristic of the exceptional extensions of COP, 
so the biorthogonal systems 
\[
(P^{(1,l_0,n)}(z))_{n\in\{0,1,2,\ldots\}\backslash\{l_0\}}, \quad
(Q^{(1,l_0,n)}(z))_{n\in\{0,1,2,\ldots\}\backslash\{l_0\}};
\]
\[
(P^{(2,l_0,n)}(z))_{n\in\{0,1,2,\ldots\}}, \quad
(Q^{(2,l_0,n)}(z))_{n\in\{0,1,2,\ldots\}};
\]
\[
(P^{(3,l_0,n)}(z))_{n\in\{0,1,2,\ldots\}}, \quad
(Q^{(3,l_0,n)}(z))_{n\in\{0,1,2,\ldots\}};
\]
\[
(P^{(4,l_0,n)}(z))_{n\in\{-l_0-1, 0,1,2,\ldots\}}, \quad
(Q^{(4,l_0,n)}(z))_{n\in\{-l_0-1, 0,1,2,\ldots\}},
\]
can be seen as exceptional extensions of HR polynomials (obtained from single-step Darboux transformation). 
We call them to type 1, type 2, type 3, and type 4 exceptional HR polynomials, respectively. 
It is again easily seen that state-deletion occurs for type 1 while state-addition occurs for type 4. 
Their expressions in terms of the HR polynomials with transformed parameters are: 
\begin{align}
\label{XP01}
&
P^{(1,l_0,n)}(z)
=nP_{l_0}(z)P_{n-1}(z;\alpha+1,\beta)-l_0P_{l_0-1}(z;\alpha+1,\beta)P_{n}(z), \\
\label{XP02}
&
P^{(2,l_0,n)}(z)
=(1-z)\bigg[nP_{l_0}(z; -\beta,-\alpha)P_{n-1}(z;\alpha+1,\beta) \\
&\hspace{24mm}
-l_0P_{l_0-1}(z; 1-\beta,-\alpha)P_{n}(z)\bigg] 
-(\alpha+\beta)P_{l_0}(z; -\beta,-\alpha)P_{n}(z), \nonumber \\
\label{XP03}
&
P^{(3,l_0,n)}(z)
=\dfrac{(\beta)_{l_0}}{(\alpha+1)_{l_0}}\bigg[
nzP_{l_0}(z; \beta-1,\alpha+1)P_{n-1}(z; \alpha+1,\beta) \\
&\hspace{8mm}
+\dfrac{(\alpha+1)l_0}{\beta+l_0-1}P_{l_0-1}(z; \beta-1,\alpha+2)P_{n}(z) 
+(1+\alpha)P_{l_0}(z; \beta-1,\alpha+1)P_{n}(z) \bigg],  \nonumber \\
\label{XP04}
&
P^{(4,l_0,n)}(z)
=
\dfrac{(-\alpha)_{l_0}}{(-\beta+1)_{l_0}}\bigg\{ 
(z-1)\bigg[nzP_{l_0}(z; -\alpha-1,-\beta+1)P_{n-1}(z;\alpha+1,\beta) \\
&\hspace{24mm}
+\dfrac{(-\beta+1)l_0}{-\alpha+l_0-1}P_{l_0-1}(z; -\alpha-1,-\beta+2)P_n(z)
\bigg]  \nonumber \\
&\hspace{8mm}
-(1-\beta-(1+\alpha)z)P_{l_0}(z; -\alpha-1,-\beta+1)P_n(z) 
\bigg\}, \nonumber 
\end{align}
where the formulas (\ref{formula_HR_1}) and (\ref{formula_HR_2}) are used. 

\begin{remark}
These four types of exceptional HR polynomials are different from each other, 
except for some special choices of seed functions. 
For example, when $l_0=1$, the relations between them turn out to be: 
\begin{align}
&
P^{(4,1,n)}(z)
=\dfrac{\alpha(\alpha+n+1)}{(n+1)(\beta-1)}P^{(1,1,n+2)}(z; \alpha-1,\beta-1), \quad n=0,1,2,\ldots \\
&
P^{(2,1,n)}(z)
=-\dfrac{\alpha(\alpha+\beta+n-1)}{(\alpha+n)(\beta-1)}P^{(3,1,n)}(z; \alpha-1,\beta-1), \quad n=0,1,2,\ldots
\end{align}
\end{remark}

\subsection{Biorthogonality}

The biorthogonality relation (\ref{Xbiorth}) can be rewritten as follows
\begin{align}
\label{nXbiorth}
&\langle \hat{\psi}(z), \widehat{L}_2^{\ast}\widehat{\psi^{\ast}}(z) \rangle
=\langle \hat{\psi}^{(j_0,l_0,n)}(z), r^{(j_0,l_0)}(z)Q^{(j_0,l_0,m)}(z) \rangle  \\
&\hspace{6mm}
=\langle w^{(j_0,l_0)}(z)P^{(j_0,l_0,n)}(z), Q^{(j_0,l_0,m)}(z) \rangle 
=(\theta_{n}-\theta^{(j_0)}_{l_0})\langle P_n(z), \overline{L_2^{\ast}}\phi^{(1,m)\ast}(z) \rangle,  \nonumber
\end{align}
where $\psi^{\ast}(z)$ was chosen as $\phi^{(1,m)\ast}(z)$ for a nonnegative integer $m$, 
the inner product $\langle \cdot , \cdot \rangle$ is defined by (\ref{inner_0}), 
$r^{(j_0,l_0)}(z)$ are the rational factors in the right-hand side of (\ref{hL2psi}), 
and the weight functions are given by $w^{(j_0,l_0)}(z)=r^{(j_0,l_0)}(z^{-1})$: 
\begin{align}
\label{XW1}
w^{(j_0,l_0)}(z)
=
\begin{cases}
\dfrac{(\beta)_{l_0}}{(1+\alpha)_{l_0}}
\dfrac{z^{l_0}(z-1)w(z)}{P^2_{l_0}(z)}, & j_0=1, \\
\dfrac{(-\alpha)_{l_0}}{(-\beta+1)_{l_0}}
\dfrac{z^{1+l_0}w(z)}{(1-z)P^2_{l_0}(z; -\beta,-\alpha)}, & j_0=2, \\
\dfrac{(1+\alpha)_{l_0}}{(\beta)_{l_0}}
\dfrac{z^{l_0}(z-1)w(z)}{P^2_{l_0}(z; \beta-1,\alpha+1)}, & j_0=3, \\
\dfrac{(-\beta+1)_{l_0}}{(-\alpha)_{l_0}}
\dfrac{z^{1+l_0}w(z)}{(1-z)P^2_{l_0}(z; -\alpha-1,-\beta+1)}, & j_0=4. 
\end{cases}
\end{align}
Here the relation (\ref{w_1_w}) is used. 
Then, using (\ref{PLP}) we explicitly provide the biorthogonality relations. 

\begin{theorem}
The biorthogonality relations for the four types of exceptional HR polynomials defined by (\ref{XP01})-(\ref{XP04}) 
and (\ref{rela_X1PQ}) are: 
\begin{align}
\label{Xbiorth1}
&\dfrac{1}{2\pi}\int_{0}^{2\pi}w^{(j_0,l_0)}(e^{ix})P^{(j_0,l_0,m)}(e^{ix})Q^{(j_0,l_0,n)}(e^{-ix})dx
=h^{(j_0,l_0)}_{n}\delta_{mn}, 
\end{align}
with
\begin{align}
\label{Xh1}
h^{(j_0,l_0)}_{n}
=
\begin{cases}
-(n+\beta)(n-l_0)h_n,  & j_0=1, \quad (n\neq l_0) \\
-(n+\beta)(n-l_0+\alpha+\beta)h_n, & j_0=2, \\
-(n+\beta)(n+l_0+1+\alpha+\beta)h_n, & j_0=3, \\
-(n+\beta)(n+l_0+1)h_n,  & j_0=4, \quad (n\neq -l_0-1)
\end{cases}
\end{align}
where the constants $h_n$ are defined by (\ref{w_HR}). 
\end{theorem}

\begin{remark}
Note that (\ref{nXbiorth}) holds under the assumption 
that the operator $\overline{ \cF^{\ast}}$ defined by (\ref{XF}) is the adjoint of $\cF$ defined by (\ref{opF}). 
First, it follows from Lemma \ref{adjoint_ness} that $\overline{ \cF^{\ast}}$ is the formal adjoint of $\cF$. 
For any $f(z), g(z)\in C^{2}$, we have 
\begin{align*}
\langle f(z), \cF g(z) \rangle - \langle \overline{ \cF^{\ast}} f(z), g(z) \rangle 
=\dfrac{1}{2\pi}\bigg[\dfrac{ie^{ix}}{\epsilon(e^{-ix})}f(e^{ix})g(e^{-ix})
\bigg]\bigg|_{0}^{2\pi}.
\end{align*}
Since $\epsilon(z)$ (see (\ref{eps})) is a Laurent polynomial, 
the right-hand side equals 0 if and only if 
\[
f(e^{2\pi i})g(e^{-2\pi i})=f(1)g(1). 
\]
Fortunately, both of $P_n(z)$ and 
$\kappa(\overline{\cG_1^{\ast}}-\overline{\cG_2^{\ast}})\tilde{\phi}_2(z)\phi^{(1,m)\ast}(z)$ 
satisfy this condition. So that (\ref{nXbiorth}) can be deduced from 
\[
\langle \cF P_n(z), \kappa(\overline{\cG_1^{\ast}}-\overline{\cG_2^{\ast}})\tilde{\phi}_2(z)\phi^{(1,m)\ast}(z) \rangle
=\langle P_n(z), \kappa\overline{ \cF^{\ast}}(\overline{\cG_1^{\ast}}-\overline{\cG_2^{\ast}})\tilde{\phi}_2(z)
\phi^{(1,m)\ast}(z) \rangle. 
\]
\end{remark}

The positive-definiteness of the exceptional weight functions can be observed from the expressions of 
$h^{(j_0,l_0)}_{n}$. 
In fact, under the conditions (\ref{cond_conv_w}) which guarantee the positivity of $h_n$, 
it is obvious that when $j_0=1,2$, $h^{(j_0,l_0)}_{n}$ is neither identically positive nor identically negative. 
However, $-h^{(3,l_0)}_{n}$ is identically positive under the conditions 
$\beta>0, \alpha+\beta>-l_0-1$, 
and $-h^{(4,l_0)}_{n}$ is identically positive under the condition
$\beta>0$. 
So, together with (\ref{cond_conv_w}), 
both of the exceptional weight functions $-w^{(3,l_0)}(z)$ and $-w^{(4,l_0)}(z)$ are positive-definite 
under the conditions 
\begin{align}
\label{pd_34}
\beta>0, \quad 
\alpha>-1, \quad
\alpha+\beta>-1. 
\end{align}

\begin{remark}
In the case of state-addition, i.e., 
the type 4 exceptional HR polynomials and biorthogonal partners whose starting terms are 
\[
P^{(4,l_0, -l_0-1)}(z)=1, \quad 
Q^{(4,l_0, -l_0-1)}(z)=1,
\]
the biorthogonality constants can be given by
\begin{align*}
&
h^{(j_0,l_0)}_{-l_0-1}
=
\langle w^{(4,l_0)}(z)P^{(4,l_0, -l_0-1)}(z), Q^{(4,l_0, -l_0-1)}(z) \rangle =
\dfrac{1}{2\pi}\int_{0}^{2\pi}w^{(4,l_0)}(e^{ix})dx  \\
&\hspace{24mm}
=C(l_0, \alpha, \beta)
\frac{\pi\Gamma(2-\beta)}{\Gamma(\alpha+1)\Gamma(1-\alpha-\beta)\sin(\pi(1-\alpha-\beta))}, 
\end{align*}
where $C(l_0, \alpha, \beta)$ is some function in $l_0, \alpha, \beta$. 
\end{remark}

\begin{remark}
To guarantee the convergence of the integral in the left-hand-side of the inner product, 
the denominators in the exceptional weight functions should not be zero. 
For type 1 exceptional HR polynomials, 
it requires that $P_{l_0}^2(z)\neq 0$ on the unit circle. 
Since $\alpha, \beta$ are real parameters, it is equivalent to require that $P_{l_0}(z)\neq 0$ for $z=\pm 1$, 
which leads to the conditions
\begin{equation}
\label{cond_X1w_1}
\alpha+\beta\neq -N, \quad 1\leq N\leq l_0, \quad 
\beta\neq \alpha+1, 
\end{equation}
where $N$ is an integer. 
Similarly, for type 2, type 3, and type 4 exceptional HR polynomials, 
the conditions are 
\begin{align}
\label{cond_X1w_2}
&
\alpha+\beta\neq N, \quad 1\leq N\leq l_0, \quad 
\beta\neq \alpha+1, \\
\label{cond_X1w_3}
&
\alpha+\beta\neq -N, \quad 1\leq N\leq l_0, \quad 
\beta\neq \alpha+1, \\
\label{cond_X1w_4}
&
\alpha+\beta\neq N, \quad 1\leq N\leq l_0, \quad 
\beta\neq \alpha+1, 
\end{align}
respectively. 

\end{remark}

\section{Multiple-step Darboux transformations on the generalized eigenvalue problem related to the HR polynomials} 
The results of single-step Darboux transformations of the GEVP related to the HR polynomials and 
the transformed eigenfunctions have been presented in the previous section. 
Here we also want to provide some observations on the results of 
multiple-step Darboux transformations and the transformed eigenfunctions. 
Below we denote the original operators by $L_1^{(0)}:=L_1$ and $L_2^{(0)}:=L_2$, 
and the $j$-th Darboux transformed operators by $L_1^{(j)}$ and $L_2^{(j)}$. 
The same notations apply to 
$\overline{L^{\ast}_1}$, $\overline{L^{\ast}_2}$, 
$\cF$, $\overline{\cF^{\ast}}$, $\cG_1$, $\overline{\cG_1^{\ast}}$, $\cG_2$, $\overline{\cG_2^{\ast}}$, 
$\phi(z)$, $\phi^{\ast}(z)$, $\kappa$, $\tilde{\phi}(z)$, 
$\psi(z)$ and $\psi^{\ast}(z)$.

First, the starting point is the original GEVP and its adjoint problem: 
\[
L_1\psi(z)=\lambda L_2\psi_n(z), \quad
\overline{L^{\ast}_1}\psi^{\ast}_n(z)=\tau \overline{L^{\ast}_2}\psi^{\ast}_n(z)
\]
with the factorizations 
\[
L_1=\tilde{\phi}_1(z)(\cG_1\cF+I),  \quad
\overline{L^{\ast}_1}=(\overline{\cF^{\ast}\cG^{\ast}_1}+I)\tilde{\phi}_1(z^{-1})I, 
\]
\[
L_2=\tilde{\phi}_2(z)(\cG_2\cF+I), \quad
\overline{L^{\ast}_2}=(\overline{\cF^{\ast}\cG^{\ast}_2}+I)\tilde{\phi}_2(z^{-1})I, 
\]
where the second-order differential operators $L_1, L_2$ are defined by (\ref{opL1}), (\ref{opL2}), 
$(\phi(z), \kappa)$ is an eigen-pair of the GEVP $L_1\phi(z)=\kappa L_2\phi(z)$, and
\[
\cF \phi(z)=0, \quad
\tilde{\phi}_1(z)=\dfrac{L_1\phi(z)}{\phi(z)}, \quad
\tilde{\phi}_2(z)=\dfrac{L_2\phi(z)}{\phi(z)}. 
\]
For simplicity, we denote $\tilde{\phi}_2(z)$ by $\tilde{\phi}(z)$, and then we have 
$\tilde{\phi}_1(z)=\kappa\tilde{\phi}(z)$. 
Recall that we have indicated in remark \ref{eigen_ast} that the adjoint problem has the same eigenvalues 
as the original problem, so in what follows, we shall let $\tau=\lambda$. 

A single-step Darboux transformation acts as isospectral transformations of the form: 
\[
L_1\psi(z)=\lambda L_2\psi(z)
\xrightarrow[]{\psi^{(1)}(z)=\cF \psi(z)}
L^{(1)}_1\psi^{(1)}(z)=\lambda L^{(1)}_2\psi^{(1)}(z), 
\]
\[
\overline{L^{\ast}_1}\psi^{\ast}(z)=\lambda \overline{L^{\ast}_2}\psi^{\ast}(z)
\xrightarrow[]{\psi^{(1)\ast}(z)=(\kappa\overline{\cG^{\ast}_1}-\lambda\overline{\cG^{\ast}_2})\tilde{\phi}(z^{-1})\psi^{\ast}(z)}
\overline{L^{\ast}_1}^{(1)} \psi^{(1)\ast}(z)=\lambda\overline{L^{\ast}_2}^{(1)} \psi^{(1)\ast}(z), 
\]
where 
\[
L^{(1)}_1=\kappa(\cF\cG_1+I), \quad
\overline{L^{\ast}_1}^{(1)}=\kappa(\overline{\cG^{\ast}_1\cF^{\ast}}+I),
\]
\[
L^{(1)}_2=\cF\cG_2+I,  \quad
\overline{L^{\ast}_2}^{(1)}=\overline{\cG^{\ast}_2\cF^{\ast}}+I. 
\]

\begin{lemma}
\label{inter_0}
The following intertwining relations can be obtained from the factorizations: 
\begin{eqnarray*}
\cF\tilde{\phi}(z)^{-1}(L_1-\lambda L_2)
\hspace{-2.4mm}&=&\hspace{-2.4mm}
(L^{(1)}_1-\lambda L^{(1)}_2)\cF,  \\
(L_1-\lambda L_2)(\kappa\cG_1-\lambda \cG_2)
\hspace{-2.4mm}&=&\hspace{-2.4mm}
\tilde{\phi}(z)(\kappa\cG_1-\lambda \cG_2)(L^{(1)}_1-\lambda L^{(1)}_2), \\
(\overline{L^{\ast}_1}-\lambda \overline{L^{\ast}_2})\tilde{\phi}(z^{-1})^{-1}\overline{\cF^{\ast}}
\hspace{-2.4mm}&=&\hspace{-2.4mm}
\overline{\cF^{\ast}}(\overline{L^{\ast}_1}^{(1)}-\lambda \overline{L^{\ast}_2}^{(1)}),  \\
(\kappa\overline{\cG^{\ast}_1}-\lambda \overline{\cG^{\ast}_2})(\overline{L^{\ast}_1}-\lambda \overline{L^{\ast}_2})
\hspace{-2.4mm}&=&\hspace{-2.4mm}
(\overline{L^{\ast}_1}^{(1)}-\lambda \overline{L^{\ast}_2}^{(1)})(\kappa\overline{\cG^{\ast}_1}-\lambda \overline{\cG^{\ast}_2})
\tilde{\phi}(z^{-1}), 
\end{eqnarray*}
where $\lambda$ can be an arbitrary constant. 
\end{lemma}
Moreover, these intertwining relations can be lifted to higher orders step by step. 
\begin{lemma}
\label{inter_1}
For $j=0,1,\ldots$, the following intertwining relations hold:  
\begin{eqnarray*}
\cF^{(j)}\tilde{\phi}^{(j)}(z)^{-1}(L^{(j)}_1-\lambda L^{(j)}_2)
\hspace{-2.4mm}&=&\hspace{-2.4mm}
(L^{(j+1)}_1-\lambda L^{(j+1)}_2)\cF^{(j)},  \\
(L^{(j)}_1-\lambda L^{(j)}_2)(\kappa^{(j)}\cG^{(j)}_1-\lambda \cG^{(j)}_2)
\hspace{-2.4mm}&=&\hspace{-2.4mm}
\tilde{\phi}^{(j)}(z)(\kappa^{(j)}\cG^{(j)}_1-\lambda \cG^{(j)}_2)(L^{(j+1)}_1-\lambda L^{(j+1)}_2), \\
(\overline{L^{\ast}_1}^{(j)}-\lambda \overline{L^{\ast}_2}^{(j)})\tilde{\phi}^{(j)}(z^{-1})^{-1}\overline{\cF^{\ast}}^{(j)}
\hspace{-2.4mm}&=&\hspace{-2.4mm}
\overline{\cF^{\ast}}^{(j)}(\overline{L^{\ast}_1}^{(j+1)}-\lambda \overline{L^{\ast}_2}^{(j+1)}),  \\
(\kappa^{(j)}\overline{\cG^{\ast}_1}^{(j)}-\lambda \overline{\cG^{\ast}_2}^{(j)})
(\overline{L^{\ast}_1}^{(j)}-\lambda \overline{L^{\ast}_2}^{(j)})
\hspace{-2.4mm}&=&\hspace{-2.4mm}
(\overline{L^{\ast}_1}^{(j+1)}-\lambda \overline{L^{\ast}_2}^{(j+1)})
(\kappa^{(j)}\overline{\cG^{\ast}_1}^{(j)}-\lambda \overline{\cG^{\ast}_2}^{(j)})
\tilde{\phi}^{(j)}(z^{-1}), 
\end{eqnarray*}
where $\lambda$ can be an arbitrary constant. 
\end{lemma}

Inductively, for $N\geq 1$, the $N$-step Darboux transformation can be characterized by the following intertwining relations: 
\begin{eqnarray*}
\dfrac{\cF^{(N-1)}}{\tilde{\phi}^{(N-1)}(z)}\circ\cdots\circ\dfrac{\cF^{(0)}}{\tilde{\phi}^{(0)}(z)}\circ
(L_1^{(0)}-\lambda L_2^{(0)})
\hspace{-2.4mm}&=&\hspace{-2.4mm}
(L^{(N)}_1-\lambda L^{(N)}_2)\circ\cF^{(N-1)}\circ\cdots\circ\cF^{(0)},  \\
&&\hspace{-70mm}
(L_1^{(0)}-\lambda L_2^{(0)})(\kappa\cG_1^{(0)}-\lambda \cG_2^{(0)})\circ\cdots
\circ(\kappa^{(N-1)}\cG_1^{(N-1)}-\lambda \cG_2^{(N-1)}) \\
&&\hspace{-72mm}=
\tilde{\phi}^{(0)}(z)(\kappa\cG_1^{(0)}-\lambda \cG_2^{(0)})\circ\cdots\circ
\tilde{\phi}^{(N-1)}(z)(\kappa^{(N-1)}\cG_1^{(N-1)}-\lambda \cG_2^{(N-1)})
(L^{(N)}_1-\lambda L^{(N)}_2), \\
(\overline{L^{\ast}_1}^{(0)}-\lambda \overline{L^{\ast}_2}^{(0)})\circ
\dfrac{\overline{\cF^{\ast}}^{(0)}}{\tilde{\phi}^{(0)}(z^{-1})}\circ\cdots\circ
\dfrac{\overline{\cF^{\ast}}^{(N-1)}}{\tilde{\phi}^{(N-1)}(z^{-1})}
\hspace{-2.4mm}&=&\hspace{-2.4mm}
\overline{\cF^{\ast}}^{(0)}\circ\cdots\circ\overline{\cF^{\ast}}^{(N-1)}\circ
(\overline{L^{\ast}_1}^{(N)}-\lambda \overline{L^{\ast}_2}^{(N)}),  \\
&&\hspace{-76mm}
(\kappa^{(N-1)}\overline{\cG^{\ast}_1}^{(N-1)}-\lambda \overline{\cG^{\ast}_2}^{(N-1)})\circ\cdots\circ
(\kappa\overline{\cG^{\ast}_1}^{(0)}-\lambda \overline{\cG^{\ast}_2}^{(0)})\circ
(\overline{L^{\ast}_1}^{(0)}-\lambda \overline{L^{\ast}_2}^{(0)}) \\
&&\hspace{-88mm}=
(\overline{L^{\ast}_1}^{(N)}-\lambda \overline{L^{\ast}_2}^{(N)})\circ
(\kappa^{(N-1)}\overline{\cG^{\ast}_1}^{(N-1)}-\lambda \overline{\cG^{\ast}_2}^{(N-1)})\tilde{\phi}^{(N-1)}(z^{-1})\circ\cdots\circ
(\kappa\overline{\cG^{\ast}_1}^{(0)}-\lambda \overline{\cG^{\ast}_2}^{(0)})\tilde{\phi}^{(0)}(z^{-1}), 
\end{eqnarray*}
where $\lambda$ can be an arbitrary constant. 
It follows from the above intertwining relations that if $\lambda$ is an eigenvalue of the original GEVP, 
then $\lambda$ is an eigenvalue of each GEVP for $j=1,2,\ldots$
In what follows, keep in mind that $\lambda$ is an eigenvalue of the GEVPs. 
Let 
\[
\cM^{(j)}_{\lambda}:=
(\kappa^{(j)}\overline{\cG^{\ast}_1}^{(j)}-\lambda \overline{\cG^{\ast}_2}^{(j)})\tilde{\phi}^{(j)}(z^{-1}),
\] 
then the $N$-step Darboux transformed eigenfunctions can be expressed as 
\begin{align}
\label{XNP_0}
&\psi^{(N)}(z)=\cF^{(N-1)}\circ\cdots\circ\cF^{(0)} [\psi(z)], \\
\label{XNP_ast_0}
&
\psi^{(N)\ast}(z)=\cM^{(N-1)}_{\lambda}\circ\cdots\circ\cM^{(0)}_{\lambda} [\psi^{\ast}(z)], \\
\label{XNQ_0}
&
\overline{L^{\ast}_2}^{(N)} [\psi^{(N)\ast}(z)]  
=\cM^{(N-1)}_{\kappa^{(N-1)}}[\psi^{(N-1)\ast}(z)] 
=\cM^{(N-1)}_{\kappa^{(N-1)}}\circ\cM^{(N-2)}_{\lambda}\circ\cdots\circ\cM^{(0)}_{\lambda} [\psi^{\ast}(z)],
\end{align}
where 
\[
\cF^{(j)}=\dfrac{\phi^{(j)}(z)}{\epsilon^{(j)}(z)}\partial_z\dfrac{1}{\phi^{(j)}(z)} I
=\dfrac{1}{\epsilon^{(j)}(z)}\left(\partial_z-(\ln{(\phi^{(j)}(z))})' I\right), \quad
\]
and $\phi^{(j)}(z)$ is the seed function in the $j$-th step Darboux transformation,  
\[
L_1^{(j)}\phi^{(j)}(z)=\kappa^{(j)}L_2^{(j)}\phi^{(j)}(z), \quad
\tilde{\phi}^{(j)}(z)=\dfrac{L_2^{(j)}\phi^{(j)}(z)}{\phi^{(j)}(z)}. 
\]

It turns out that the operators $\cF^{(j)}$ and 
$\cM^{(j)}_{\kappa^{(j)}}$, $j=0,1,\ldots$,  
can be expressed in a similar form. 
In fact, by induction, we know 
$A_2^{(j)}(z)=0$, $j=1,\ldots, N-1$, and
\begin{align*}
&\dfrac{\kappa^{(N-1)}(\overline{\cG^{\ast}_1}^{(N-1)}-\overline{\cG^{\ast}_2}^{(N-1)})\tilde{\phi}^{(N-1)}(z^{-1})}
{z^2A_1^{(N-1)}(z^{-1})\epsilon^{(N-1)}(z^{-1})}  \\
&=\partial_z
+\bigg(\dfrac{B_1^{(N-1)}(z^{-1})-(A^{(N-1)}_1)'(z^{-1})-\kappa^{(N-1)}B_2^{(N-1)}(z^{-1})}{z^2A_1^{(N-1)}(z^{-1})}
+\dfrac{1}{z}+\dfrac{(\phi^{(N-1)})'(z^{-1})}{z^2\phi^{(N-1)}(z^{-1})})
\bigg) I \\
&=\partial_z
+\left(-\dfrac{(w^{(N-1)}_0)'(z)}{w^{(N-1)}_0(z)}-\dfrac{\kappa^{(N-1)}}{z}+\dfrac{(\phi^{(N-1)})'(z^{-1})}{z^2\phi^{(N-1)}(z^{-1})}
\right) I \\
&=\partial_z
-\left(\ln{(w^{(N-1)}_0(z)z^{\kappa^{(N-1)}}\phi^{(N-1)}(z^{-1}))}
\right)' I,  \\
&\dfrac{(\kappa\overline{\cG^{\ast}_1}-\lambda \overline{\cG^{\ast}_2})\tilde{\phi}(z^{-1})}
{z^2A_1(z^{-1})\epsilon(z^{-1})}
=\partial_z
+\left(\dfrac{B_1(z^{-1})-A'_1(z^{-1})-\lambda B_2(z^{-1})}{z^2A_1(z^{-1})}
+\dfrac{1}{z}+\dfrac{\phi'(z^{-1})}{z^2\phi(z^{-1})})\right) I \\
&\hspace{33mm}
=\partial_z
+\left(-\dfrac{w'_1(z)}{w_1(z)}-\dfrac{\lambda}{z}+\dfrac{\phi'(z^{-1})}{z^2\phi(z^{-1})}
\right) I \\
&\hspace{33mm}
=\partial_z
-\left(\ln{(w_1(z)z^{\lambda}\phi(z^{-1}))}
\right)' I,  \\
&\dfrac{(\kappa^{(j)}\overline{\cG^{\ast}_1}^{(j)}-\lambda\overline{\cG^{\ast}_2}^{(j)})\tilde{\phi}^{(j)}(z^{-1})}
{z^2A^{(j)}_1(z^{-1})\epsilon^{(j)}(z^{-1})}
=\partial_z
-\left(\ln{(w^{(j)}_1(z)z^{\lambda}\phi^{(j)}(z^{-1}))}
\right)' I, \quad j=1,\ldots, N-2,
\end{align*}
where $w^{(j)}_1(z)$, $j=1,\ldots, N-1$, are defined by 
\[
\dfrac{(w^{(j)}_1)'(z)}{w^{(j)}_1(z)}
=-\dfrac{1}{z}-\dfrac{B_1^{(j)}(z^{-1})-(A^{(j)}_1)'(z^{-1})}{z^2A_1^{(j)}(z^{-1})}. 
\]

Note that for any eigenpair $(\phi^{(j)\ast}(z), \kappa^{(j)})$ of the GEVP 
\[
((\overline{L^{\ast}_1}^{(j)}-\kappa^{(j)}\overline{L^{\ast}_2}^{(j)}))\phi^{(j)\ast}(z)=0,
\] 
it follows from the last intertwining relation of lemma \ref{inter_1} that 
\[
\cM^{(j)}_{\kappa^{(j)}}\phi^{(j)\ast}(z)=0, 
\]
which implies 
\[
\phi^{(j)\ast}(z)=S_nw^{(j)}_1(z)z^{\kappa^{(j)}}\phi^{(j)}(z^{-1}),  \quad
S_n\in\mathbb{R}. 
\]
Denote by $M_{j}(z)=z^2A^{(j)}_1(z^{-1})\epsilon^{(j)}(z^{-1})$, then for $j=0,1,\ldots$, we have 
\begin{align}
&
\cM^{(j)}_{\kappa^{(j)}}
=M_{j}(z)\left(\partial_z
-\left(\ln{(\phi^{(j)\ast}(z))}
\right)' I
\right), \\
&
\cM^{(j)}_{\lambda}
=M_{j}(z)\left(\partial_z
-\left(\ln{(z^{\lambda-\kappa^{(j)}}\phi^{(j)\ast}(z))}
\right)' I
\right).
\end{align}

By choosing the seed functions from the transformed quasi-polynomial eigenfunctions 
\[
\phi^{(0)}(z)=\phi^{(k_0,l_0)}(z), \quad 
\phi^{(1)}(z)=\cF[\phi^{(k_1,l_1)}(z)], \quad
\phi^{(2)}(z)=\cF^{(1)}\circ\cF[\phi^{(k_2,l_2)}(z)], \quad \cdots
\] 
where $k_0,k_1,k_2,\ldots\in\{1,2,3,4\}$, and $l_0, l_1, l_2, \cdots\in\mathbb{N}_{\geq 1}$, 
the $N$-step Darboux transformed eigenfunctions (\ref{XNP_0}) can be expressed in terms of Wronski determinants: 
\begin{align}
\psi^{(N)}(z)
=\left(\prod^{N-1}_{j=0}\epsilon^{(j)}(z)\right)^{-1}
\dfrac{Wr[\phi^{(k_0,l_0)}(z), \cdots, \phi^{(k_{N-1},l_{N-1})}(z), \psi(z)]}
{Wr[\phi^{(k_{0},l_0)}(z), \cdots, \phi^{(k_{N-1},l_{N-1})}(z)]},
\end{align}
where
\[
Wr[y_1(z), \cdots, y_{n-1}(z), y_n(z)] 
=\begin{vmatrix}
y_1(z) & \cdots & y_{n-1}(z) & y_n(z) \\
y'_1(z) & \cdots & y'_{n-1}(z) & y'_n(z) \\
\vdots & \ddots & \vdots & \vdots \\
y^{(n)}_1(z) & \cdots & y^{(n)}_{n-1}(z) & y^{(n)}_n(z) \\
\end{vmatrix}.
\]

At the same time, the product (\ref{XNQ_0}) of the biorthogonal partners and the conjugate of the weight function 
can be expressed as 
\begin{align}
&\overline{L^{\ast}_2}^{(N)} [\psi^{(N)\ast}(z)]  
=\left(\prod^{N-1}_{j=0}M_{j}(z)\right) \\
&\hspace{6mm}\cdot
\dfrac{Wr[z^{\lambda-\kappa^{(0)}}\phi^{(k_0,l_0)\ast}(z), \cdots, 
z^{\lambda-\kappa^{(N-2)}}\phi^{(k_{N-2},l_{N-2})\ast}(z), 
\phi^{(k_{N-1},l_{N-1})\ast}(z), \psi^{\ast}(z)]}
{Wr[z^{\lambda-\kappa^{(0)}}\phi^{(k_0,l_0)\ast}(z), \cdots, z^{\lambda-\kappa^{(N-2)}}\phi^{(k_{N-2},l_{N-2})\ast}(z), 
\phi^{(k_{N-1},l_{N-1})\ast}(z)]}, \nonumber
\end{align}
The biorthogonality relation holds for
\begin{align}
&\langle \psi^{(N)}(z), \overline{L^{\ast}_2}^{(N)} [\psi^{(N)\ast}(z)]   \rangle
=(\lambda-\kappa^{(N-1)})\cdots(\lambda-\kappa^{(0)})
\langle \psi^{(0)}(z), \overline{L^{\ast}_2}^{(0)} [\psi^{(0)\ast}(z)]   \rangle.  \nonumber
\end{align}

As for the classification of the exceptional HR polynomials obtained from multiple-step Darboux transformations, 
the positive-definiteness of the weight functions and the related issues, 
we will discuss them carefully in our future work.

\section{Concluding remarks}

The theory of exceptional orthogonal polynomials started around 2008 
and has developed rapidly since then.
The exceptional extensions of COP from the Askey scheme and Askey-Wilson scheme have been studied widely and deeply, 
including the $q\rightarrow -1$ cases \cite{XBI}.
However, exceptional extensions to more general orthogonal functions related to GEVP
have not been attempted before.
In this paper, we construct exceptional extensions of the HR polynomials as a first example
of satisfying the GEVP for a pair of differential operators.

There are similarities between exceptional analogs of LBP and COP. 
For instance, both of them allow gaps in their degree sequences; in both cases, the state-deletion and state-addition occur. But one can find that there exists a significant difference in the appearance of quasi-Laurent-polynomial eigenfunctions of the GEVP related to an LBP, which has not appeared in the case of COP. 
In our upcoming series of work, the exceptional Pastro polynomials 
(which are the $q$-analogue of HR's case) 
will be provided, and the exceptional biorthogonal polynomials obtained from multiple-step Darboux transformations will be discussed in detail. 
Many other important properties such as recurrence relations for the exceptional biorthogonal functions
are left for future work.

\section*{Acknowledgement}
This work was partially supported by JSPS KAKENHI Grant Numbers JP19H0179.



\begin{thebibliography}{}

\bibitem{A82}
Askey R.: 
Discussion of Szeg\"o's paper ``Beiträge zur Theorie der Toeplitzschen Formen'', 
Gabor Szeg\"o. Collected works, {\bf{1}}, 303–305 (1982)

\bibitem{A85}
Askey R.: 
Some problems about special functions and computations, 
Rend. Semin. Mat. Univ. Politec. Torino, 1-22 (1985)

\bibitem{BispOP}
Dur\'an A. J., de la Iglesia M. D.: 
Constructing bispectral orthogonal polynomials from the classical discrete families of Charlier, 
Meixner and Krawtchouk, 
Constr. Approx. {\bf{41}}(1), 49-91 (2015)

\bibitem{Krall_Hahn}
Dur\'an A. J., Manuel D.: 
Constructing Krall–Hahn orthogonal polynomials, 
J. Math. Anal. Appl. {\bf{424}}(1), 361-384 (2015)

\bibitem{XOPviaKrall}
Dur\'an A. J.: 
Exceptional orthogonal polynomials via Krall discrete polynomials, 
Lectures on Orthogonal Polynomials and Special Functions, {\bf{464}}, 1, (2020).

\bibitem{X-Bochner}
García-Ferrero M. \'A., G\'omez-Ullate D., Milson R.: 
A Bochner type characterization theorem for exceptional orthogonal polynomials, 
J. Math. Anal. Appl. {\bf{472}}(1): 584-626 (2019)

\bibitem{GKM09}
G\'omez-Ullate D., Kamran N., Milson R.: 
An extended class of orthogonal polynomials defined by a Sturm-Liouville problem, 
J. Math. Anal. Appl. {\bf{359}}, 352–367 (2009)

\bibitem{GKM10_1}
G\'omez-Ullate D., Kamran N., Milson R.: 
An extension of Bochner's problem: exceptional invariant subspaces,
J. Approx. Theor. {\bf{162}}(5), 987-1006 (2010)

\bibitem{GKM10_2}
G\'omez-Ullate D., Kamran N., Milson R.: 
Exceptional orthogonal polynomials and the Darboux transformation, 
J. Phys. A Math. Theor. {\bf{43}}(43): 434016 (2010)

\bibitem{GKM12}
G\'omez-Ullate D., Kamran N., Milson R.: 
Two-step Darboux transformations and exceptional Laguerre polynomials, 
J. Math. Anal. Appl. {\bf{387}}(1), 410-418 (2012)

\bibitem{GGM13}
G\'omez-Ullate D., Grandati Y., Milson R.: 
Rational extensions of the quantum harmonic oscillator and exceptional Hermite polynomials, 
J. Phys. A Math. Theor. {\bf{47}}(1): 015203 (2013)

\bibitem{GKK16}
G\'omez-Ullate D., Kasman A., Kuijlaars A. B. J., Milson R.: 
Recurrence relations for exceptional Hermite polynomials, 
J. Approx. Theor. {\bf{204}}, 1-16 (2016)


\bibitem{GMM13}
G\'omez-Ullate D., Marcell\'an F., Milson R.: 
Asymptotic and interlacing properties of zeros of exceptional Jacobi and Laguerre polynomials, 
J. Math. Anal. Appl. {\bf{399}}(2), 480-495 (2013)

\bibitem{GVZ04}
Gr\"unbaum F. A., Vinet L., Zhedanov A.: 
Linear operator pencils on Lie algebras and Laurent biorthogonal polynomials, 
J. Phys. A Math. Gen. {\bf{37}}(31), 7711 (2004)

\bibitem{HR_1}
Hendriksen E., Njåstad O.: 
Biorthogonal Laurent polynomials with biorthogonal derivatives, 
Rocky. MT. J. Math. 301-317 (1991)

\bibitem{HR}
Hendriksen E., van Rossum H.: 
Orthogonal Laurent polynomials, 
Indagationes Mathematicae (Proceedings). North-Holland, {\bf{89}}(1), 17-36 (1986)

\bibitem{HS12}
Ho C. L., Sasaki R.: 
Zeros of the exceptional Laguerre and Jacobi polynomials, 
International Scholarly Research Notices 2012 (2012)


\bibitem{KM15}
Kuijlaars A. B. J., Milson R.: 
Zeros of exceptional Hermite polynomials, 
J. Approx. Theor. {\bf{200}}, 28-39 (2015)

\bibitem{LLM16}
Liaw C., Littlejohn L. L., Milson R., Stewart, J.: 
The spectral analysis of three families of exceptional Laguerre polynomials, 
J. Approx. Theor. {\bf{202}}, 5-41 (2016)

\bibitem{XBI}
Luo Y., Tsujimoto S.: 
Exceptional Bannai–Ito polynomials, 
J. Approx. Theor. {\bf{239}}, 144-173 (2019)

\bibitem{MR09}
Midya B., Roy B.: 
Exceptional orthogonal polynomials and exactly solvable potentials in position dependent mass Schrödinger Hamiltonians, 
Phys. Lett. A, {\bf{373}}(45), 4117-4122 (2009)

\bibitem{X-RR}
Miki H., Tsujimoto S.: 
A new recurrence formula for generic exceptional orthogonal polynomials, 
J. Math. Phys. {\bf{56}}(3): 033502 (2015)

\bibitem{X-Kra}
Miki H., Tsujimoto S., Vinet L.: 
The single-indexed exceptional Krawtchouk polynomials, 
arXiv: 2201.12359[math.CA] (2022)

\bibitem{OS09}
Odake S., Sasaki R.: 
Infinitely many shape invariant discrete quantum mechanical systems and new exceptional orthogonal polynomials related to the Wilson and Askey-Wilson polynomials, 
Phys. Lett. B, {\bf{682}}(1): 130-136 (2009)

\bibitem{OS11}
Odake S., Sasaki R.: 
Exactly solvable quantum mechanics and infinite families of multi-indexed orthogonal polynomials, 
Phys. Lett. B, {\bf{702}}(2-3), 164-170 (2011) 

\bibitem{X q-Racah}
Odake S., Sasaki R.: 
The exceptional ($X_l$)($q$)-Racah polynomials, 
Progr. Theoret. Phys. {\bf{125}}(5), 851-870 (2011)

\bibitem{OS13}
Odake S., Sasaki R.: 
Multi-indexed Wilson and Askey-Wilson polynomials, 
J. Phys. A Math. Theor. {\bf{46}}(4), 045204 (2013)

\bibitem{O16}
Odake S.: 
Recurrence relations of the multi-indexed orthogonal polynomials. III, 
J. Math. Phys. {\bf{57}}(2), 023514 ( 2016)

\bibitem{Quesne08}
Quesne C.: 
Exceptional orthogonal polynomials, exactly solvable potentials and supersymmetry, 
J. Phys. A Math. Theor. {\bf{41}}(39), 392001 (2008)

\bibitem{Quesne09}
Quesne C.: 
Solvable rational potentials and exceptional orthogonal polynomials in supersymmetric quantum mechanics, 
SIGMA Symmetry Integrability Geom. Methods Appl. 5: 084 (2009)

\bibitem{STA}
Sasaki R., Tsujimoto S., Zhedanov A.: 
Exceptional Laguerre and Jacobi polynomials and the corresponding potentials through Darboux-Crum transformations, 
J. Phys. A Math. Theor. {\bf{43}}, 315204 (2010)

\bibitem{ABP}
Vinet L., Zhedanov A.: 
An algebraic treatment of the Askey biorthogonal polynomials on the unit circle, 
Forum of Mathematics, Sigma. Cambridge University Press, {\bf{9}} (2021)

\bibitem{CLBP}
Zhedanov A.: 
The ``classical'' Laurent biorthogonal polynomials, 
J. Comput. Appl. Math. {\bf{98}}(1), 121-147 (1998)

\bibitem{BORF_GEVP}
Zhedanov A.: 
Biorthogonal rational functions and the generalized eigenvalue problem, 
J. Approx. Theor. {\bf{101}}(2), 303-329 (1999)

\bibitem{PASI}
Zhedanov A.: 
On the polynomials orthogonal on regular polygons, 
J. Approx. Theor. {\bf{97}}(1), 1-14 (1999)
\end{thebibliography}
\end{document}